\documentclass[12pt]{amsart}
\usepackage{amssymb}
\begin{document}
\numberwithin{equation}{section}

\def\1#1{\overline{#1}}
\def\2#1{\widetilde{#1}}
\def\3#1{\widehat{#1}}
\def\4#1{\mathbb{#1}}
\def\5#1{\frak{#1}}
\def\6#1{{\mathcal{#1}}}

\newcommand{\UH}{\mathbb{H}}
\newcommand{\de}{\partial}
\newcommand{\R}{\mathbb R}
\newcommand{\Ha}{\mathbb H}
\newcommand{\al}{\alpha}
\newcommand{\tr}{\widetilde{\rho}}
\newcommand{\tz}{\widetilde{\zeta}}
\newcommand{\tk}{\widetilde{C}}
\newcommand{\tv}{\widetilde{\varphi}}
\newcommand{\hv}{\hat{\varphi}}
\newcommand{\tu}{\tilde{u}}
\newcommand{\tF}{\tilde{F}}
\newcommand{\debar}{\overline{\de}}
\newcommand{\Z}{\mathbb Z}
\newcommand{\C}{\mathbb C}
\newcommand{\Po}{\mathbb P}
\newcommand{\zbar}{\overline{z}}
\newcommand{\G}{\mathcal{G}}
\newcommand{\So}{\mathcal{S}}
\newcommand{\Ko}{\mathcal{K}}
\newcommand{\U}{\mathcal{U}}
\newcommand{\B}{\mathbb B}
\newcommand{\oB}{\overline{\mathbb B}}
\newcommand{\Cur}{\mathcal D}
\newcommand{\Dis}{\mathcal Dis}
\newcommand{\Levi}{\mathcal L}
\newcommand{\SP}{\mathcal SP}
\newcommand{\Sp}{\mathcal Q}
\newcommand{\A}{\mathcal O^{k+\alpha}(\overline{\mathbb D},\C^n)}
\newcommand{\CA}{\mathcal C^{k+\alpha}(\de{\mathbb D},\C^n)}
\newcommand{\Ma}{\mathcal M}
\newcommand{\Ac}{\mathcal O^{k+\alpha}(\overline{\mathbb D},\C^{n}\times\C^{n-1})}
\newcommand{\Acc}{\mathcal O^{k-1+\alpha}(\overline{\mathbb D},\C)}
\newcommand{\Acr}{\mathcal O^{k+\alpha}(\overline{\mathbb D},\R^{n})}
\newcommand{\Co}{\mathcal C}
\newcommand{\Hol}{{\sf Hol}}
\newcommand{\Aut}{{\sf Aut}(\mathbb D)}
\newcommand{\D}{\mathbb D}
\newcommand{\oD}{\overline{\mathbb D}}
\newcommand{\oX}{\overline{X}}
\newcommand{\loc}{L^1_{\rm{loc}}}
\newcommand{\la}{\langle}
\newcommand{\ra}{\rangle}
\newcommand{\thh}{\tilde{h}}
\newcommand{\N}{\mathbb N}
\newcommand{\kd}{\kappa_D}
\newcommand{\Hr}{\mathbb H}
\newcommand{\ps}{{\sf Psh}}
\newcommand{\Hess}{{\sf Hess}}
\newcommand{\subh}{{\sf subh}}
\newcommand{\harm}{{\sf harm}}
\newcommand{\ph}{{\sf Ph}}
\newcommand{\tl}{\tilde{\lambda}}
\newcommand{\gdot}{\stackrel{\cdot}{g}}
\newcommand{\gddot}{\stackrel{\cdot\cdot}{g}}
\newcommand{\fdot}{\stackrel{\cdot}{f}}
\newcommand{\fddot}{\stackrel{\cdot\cdot}{f}}

\def\Re{{\sf Re}\,}
\def\Im{{\sf Im}\,}

\newcommand{\Real}{\mathbb{R}}
\newcommand{\Natural}{\mathbb{N}}
\newcommand{\Complex}{\mathbb{C}}
\newcommand{\ComplexE}{\overline{\mathbb{C}}}
\newcommand{\Int}{\mathbb{Z}}
\newcommand{\UD}{\mathbb{D}}
\newcommand{\clS}{\mathcal{S}}
\newcommand{\gtz}{\ge0}
\newcommand{\gt}{\ge}
\newcommand{\lt}{\le}
\newcommand{\fami}[1]{(#1_{s,t})}
\newcommand{\famc}[1]{(#1_t)}
\newcommand{\ts}{t\gt s\gtz}
\newcommand{\classCC}{\tilde{\mathcal C}}
\newcommand{\classS}{\mathcal S}

\newcommand{\Unit}[1]{\begin{trivlist}\item\noindent{#1}}
\newcommand{\Step}[1]{\Unit{\bf Step~#1.}}
\newcommand{\step}[1]{\Unit{\underline{Step~#1}.}}
\def\endunit{\end{trivlist}}
\let\endstep=\endunit
\newcommand{\proofbox}{\hfill$\Box$}
\newcommand{\Case}[1]{\Unit{\it Case #1.}}

\newcommand{\mcite}[1]{\csname b@#1\endcsname}
\newcommand{\UC}{\mathbb{T}}

\newcommand{\Moeb}{\mathrm{M\ddot ob}}

\newcommand{\dAlg}{{\mathcal A}(\UD)}
\newcommand{\diam}{\mathrm{diam}}



\newcommand{\Spec}{\Lambda^d}
\newcommand{\SpecR}{\Lambda^d_R}
\newcommand{\Prend}{\mathrm P}
\newcommand{\classM}{\mathbb M}
\newcommand{\Log}{\mathop{\mathrm{Log}}}

\def\ch{\mathop{\mathrm{ch}}}
\def\th{\mathop{\mathrm{th}}}
\def\sh{\mathop{\mathrm{sh}}}
\def\tg{\mathop{\mathrm{tg}}}
\newcommand{\CarClass}{{\mathcal C}}
\newcommand{\ParClass}{{\mathcal V}}



\def\cn{{\C^n}}
\def\cnn{{\C^{n'}}}
\def\ocn{\2{\C^n}}
\def\ocnn{\2{\C^{n'}}}
\def\je{{\6J}}
\def\jep{{\6J}_{p,p'}}
\def\th{\tilde{h}}


\def\dist{{\rm dist}}
\def\const{{\rm const}}
\def\rk{{\rm rank\,}}
\def\id{{\sf id}}
\def\aut{{\sf aut}}
\def\Aut{{\sf Aut}}
\def\CR{{\rm CR}}
\def\GL{{\sf GL}}
\def\Re{{\sf Re}\,}
\def\Im{{\sf Im}\,}
\def\U{{\sf U}}

\def\la{\langle}
\def\ra{\rangle}

\emergencystretch15pt \frenchspacing

\newtheorem{theorem}{Theorem}[section]
\newtheorem{lemma}[theorem]{Lemma}
\newtheorem{proposition}[theorem]{Proposition}
\newtheorem{corollary}[theorem]{Corollary}

\makeatletter
\newtheorem{result}{Theorem}\setcounter{result}{64}\def\theresult{\noexpand\char\the\value{result}}
\makeatother

\theoremstyle{definition}
\newtheorem{definition}[theorem]{Definition}
\newtheorem{example}[theorem]{Example}

\theoremstyle{remark}
\newtheorem{remark}[theorem]{Remark}
\numberwithin{equation}{section}

\long\def\REM#1{\relax}

\newcommand{\func}[1]{\mathop{\mathrm{#1}}}
\renewcommand{\Re}{\func{Re}}
\renewcommand{\Im}{\func{Im}}

\newenvironment{mylist}{\begin{list}{}%
{\labelwidth=2em\leftmargin=\labelwidth\itemsep=.4ex plus.1ex minus.1ex\topsep=.7ex plus.3ex
minus.2ex}%
\let\itm=\item\def\item[##1]{\itm[{\rm ##1}]}}{\end{list}}

\title[Loewner Theory in annulus II]{Loewner Theory in annulus II: Loewner chains}

\author[M. D. Contreras]{Manuel D. Contreras$^\dag$}

\author[S. D\'{\i}az-Madrigal]{Santiago D\'{\i}az-Madrigal$^\dag$}
\address{Camino de los Descubrimientos, s/n\\
Departamento de Matem\'{a}tica Aplicada II\\
Escuela T\'{e}cnica Superior de Ingenier\'\i a\\
Universidad de Sevilla\\
Sevilla, 41092\\
Spain.}\email{contreras@us.es} \email{madrigal@us.es}

\author[P. Gumenyuk]{Pavel Gumenyuk$^\ddag$}
\address{Department of
Mathematical Sciences\\ Norwegian University of Science and Technology \\
Trondheim 7491, Norway.} \email{Pavel.Gumenyuk@math.ntnu.no}

\expandafter\def\csname subjclassname@2010\endcsname{\textup{2010}
Mathematics Subject Classification}

\date{\today }


\thanks{$^\dag\,^\ddag$ Partially supported by the ESF Networking Programme ``Harmonic and Complex Analysis
and its Applications'' and by \textit{La Consejer\'{\i}a de Econom\'{\i}a, Innovaci\'{o}n y Ciencia
de la Junta de Andaluc\'{\i}a} (research group FQM-133)}

\thanks{$^\dag$ Partially supported by the \textit{Ministerio
de Ciencia e Innovaci\'on} and the European Union (FEDER), project MTM2009-14694-C02-02}

\thanks{$^\ddag$ Supported by a grant from Iceland, Liechtenstein and Norway through the EEA Financial Mechanism.
Supported and coordinated by Universidad Complutense de Madrid and by Instituto de Matem\'{a}ticas
de la Universidad de Sevilla. Partially supported by the {\it Scandinavian Network ``Analysis
and Applications''} (NordForsk), project \#080151,%
\ and the {\it Research Council of Norway}, project \#177355/V30}

\begin{abstract}
Loewner Theory, based on dynamical viewpoint, proved itself to be a
powerful tool in Complex Analysis and its applications. Recently
Bracci et al \cite{BCM1,BCM2,SMP} have proposed a new approach
bringing together all the variants of the (deterministic) Loewner
Evolution in a simply connected reference domain. This paper is
devoted to the construction of a general version of Loewner Theory
for the annulus launched in~\cite{SMP3}. We introduce the general
notion of a Loewner chain over a system of annuli and obtain a
1-to-1 correspondence between Loewner chains and evolution families
in the doubly connected setting similar to that in the Loewner
Theory for the unit disk. Futhermore, we establish a conformal
classification of Loewner chains via the corresponding evolution
families and via semicomplete weak holomorphic vector fields.
Finally, we extend the explicit characterization of the semicomplete
weak holomorphic vector fields obtained in~\cite{SMP3} to the
general case.
\end{abstract}

\maketitle

\tableofcontents

\section{Introduction}
\newcommand{\pr}{\mathop{\mathrm{pr}}}

Loewner Theory can be regarded as a theory providing a parametric
representation of univalent functions in the unit
disk~$\UD:=\{z:|z|<1\}$ based on an infinitesimal description of the
semigroup of injective holomorphic self-maps of~$\UD$. Originating
in Loewner's paper~\cite{Loewner} of 1923, this theory gave a great
impact in the development of Complex Analisys, in which connection
one might recall, e.g., its crucial role in the proof of the famous
Bieberbach's Conjecture (see, e.g., \cite[Chapter 17]{Conway2})
given by de Branges~\cite{deBranges} in 1984.

From another point of view, Loewner Theory can be seen as an
analytic tool to describe monotonic (expanding or contracting)
domain dynamics in the plane. A stochastic version of such dynamics
(SLE), introduced by Schramm~\cite{Schramm} in 2000, is of great
importance\footnote{Which is testified by two Fields Medals, awarded
to Wendelin Werner in 2006 and to Stanislav Smirnov in 2010.}
because of its intrinsic connection to classical lattice models of
Statistical Physics such as percolation and the planar Ising model.

We note also that the well celebrated free-boundary Hele-Shaw
problem describing 2D filtration processes (see, e.g.,
\cite{GusVas}) is driven by a non-linear analogue of the classical
Loewner\,--\,Kufarev PDE, playing one of the central roles in
Loewner Theory. Finally, we would like to mention recently
discovered relations between classical Loewner Theory, integrable
systems and representation of the Virasoro algebra, which appears in
a number of fundamental problems in Mathematical Physics, see
\cite{Roland, MPV1, MPV2, VasMark-Sevilla}.

According to the new general approach in Loewner
Theory~\cite{BCM1,BCM2} introduced recently by Bracci and the first
two  authors, the essence of the modern Loewner Theory resides in
the connection and interplay between three basic notions: {\it
evolution families}, {\it Loewner chains}, and {\it semicomplete
weak holomorphic vector fields}.

This paper is a sequel of~\cite{SMP3} and devoted to the
construction of a general version of the Loewner Theory for doubly
connected domains. For details concerning the classical and
modern Loewner Theory in simply connected case and for the history
of its extension to multiply connected case we refer the reader
to~\cite{SMP3} and references cited therein. A historical survey on
Loewner Theory and related references can be also found
in~\cite{letters}.

In~\cite{SMP3} we introduced a general notion of {\it evolution family} over an
increasing continuous family of annuli and established a 1-to-1 correspondence
between these evolution families and semicomplete weak holomorphic vector
fields. In the case of all annuli being non-degenerate we also obtained an
explicit representation of the involved semicomplete weak holomorphic vector
fields.

In the present paper we will introduce a general notion of Loewner
chain in the doubly connected setting and study its relation to
evolution families and semicomplete weak holomorphic vector fields.
Moreover, we will establish a conformal classification of Loewner
chains and obtain an explicit representation of semicomplete weak
holomorphic vector fields in a more general case than the one
considered in~\cite{SMP3}, allowing the annuli to degenerate into a
punctured disk starting from some point.

\subsection{Preliminaries}\label{SS_prel}
Now we are going to introduce some definitions and results
from~\cite{SMP3} necessary for our discussion.

In comparison with the simply connected setting, a new feature in the doubly
(and more generally, multiply) connected case is that in order to develop a rich theory, instead of a static
reference domain (the unit disk) one has to consider a family of canonical
domains $(D_t)_{t\ge0}$, with evolution families being formed by holomorphic
mappings $ \varphi_{s,t}:D_s\to D_t$, $0\le s\le t$. This explains the reason
for the following definition.

Denote $\mathbb A_{r,R}:=\{z:r<|z|<R\}$, $\mathbb A_r:=\mathbb
A_{r,1}$, $0\le r< R\le+\infty$, and let $\UD^*$ stand for $\mathbb
A_{0}=\UD\setminus\{0\}$. For $d\in[1,+\infty]$ by $AC^d(X,Y)$,
$X\subset \Real$, $Y\subset\Complex$, we denote the class of all
locally absolutely continuous functions $f:X\to Y$ such that $f'\in
L^d_{\mathrm{loc}}(X,\Complex)$. Finally, let
\begin{equation}\label{EQ_omega}
\omega(r):=\left\{%
\begin{array}{ll}
   -\pi/\log r,& \text{if~}r\in(0,1),\\%
   0,& \text{if~}r=0.
\end{array}\right.
\end{equation}

\begin{definition}[\cite{SMP3}]\label{def-cansys}
Let $d\in[1,+\infty]$ and $(D_t)_{t\ge0}$ be a family of annuli
$D_t:=\mathbb A_{r(t)}$. We will say that $(D_t)$ is a {\it (doubly
connected) canonical domain system of order~$d$} (or in short, a
{\it canonical $L^d$-system}) if the function $t\mapsto
\omega(r(t))$ belongs to $AC^d\big([0,+\infty),[0,+\infty)\big)$ and
does not increase. If $r(t)\equiv0$, then the canonical domain
system~$(D_t)$ will be called {\it degenerate}. If on the contrary
$r(t)$ does not vanish, then ~$(D_t)$ will be called {\it
non-degenerate}. Finally, if there exists $T>0$ such that $r(t)>0$
for all $t\in[0,T)$ and $r(t)=0$ for all $t\ge T$, then we will say
that~$(D_t)$ is {\it of mixed type}.
\end{definition}
\begin{remark}
The condition that $t\mapsto \omega(r(t))$ is of class $AC^d$
implies that $t\mapsto r(t)$ also belongs to
$AC^d\big([0,+\infty),[0,1)\big)$. If $r(t)>0$ for all~$t\ge0$, or
$d=1$, then the converse is also true and we can replace
$\omega(r(t))$ by $r(t)$ in the above definition. However, in
general we do not know whether this is possible, because some proofs
in~\cite{SMP3} use essentially the requirement  that
$t\mapsto\omega(r(t))$ is of class~$AC^d$.
\end{remark}

Now we can introduce the definition of an evolution family for the
doubly connected setting.

\begin{definition}[\cite{SMP3}]\label{def-ev}
Let $(D_t)_{t\ge0}$ be a canonical domain system of
order~$d\in[1,+\infty]$. A family $(\varphi_{s,t})_{0\leq s\leq
t<+\infty}$ of holomorphic mappings $\varphi_{s,t}:D_s\to D_t$ is
said to be an {\it evolution family of order $d$ over~$(D_t)$} (in
short, an {\it $L^d$-evolution family}) if the following conditions
are satisfied:
\begin{mylist}
\item[EF1.] $\varphi_{s,s}=\mathsf{id}_{D_s},$

\item[EF2.] $\varphi_{s,t}=\varphi_{u,t}\circ\varphi_{s,u}$ whenever $0\le
s\le u\le t<+\infty,$

\item[EF3.] for any closed interval $I:=[S,T]\subset[0,+\infty)$ and any $z\in D_S$ there exists a
non-negative function ${k_{z,I}\in L^{d}\big([S,T],\mathbb{R}\big)}$
such that
\[
|\varphi_{s,u}(z)-\varphi_{s,t}(z)|\leq\int_{u}^{t}k_{z,I}(\xi)d\xi
\]
whenever $S\leq s\leq u\leq t\leq T.$
\end{mylist}
Suppressing the language we will refer also to the pair $\mathcal
E:=\big((D_t), (\varphi_{s,t})\big)$ as an evolution family of
order~$d$ and apply terms {\it degenerate}, {\it non-degenerate},
{\it of mixed type} to $\mathcal E$ whenever they are applicable to
the canonical domain system~$(D_t)$.
\end{definition}

The notion of weak holomorphic vector field, as introduced
in~\cite{SMP3}, in the doubly connected setting can be defined as
follows. Let $\pr_{\Real}$ stand for the projection
$\Complex\times\Real\ni(z,t)\mapsto t\in\Real$.

\begin{definition}\label{D_WHVF}
Let  $d\in[1,+\infty]$ and $(D_t)$ be a canonical domain system of order~$d$.  A function $G:\mathcal
D\to\Complex$, where $\mathcal D:=\{(z,t):t\ge0,\,\,z\in D_t\}$,  is said to be a {\it weak
holomorphic vector field} of order~$d$ over $(D_t)$, if it satisfies the following conditions:
\begin{mylist}
\item[WHVF1.] For each $z\in\Complex$ the function $G(z,\cdot)$ is measurable in~$E_z:=\{t\ge0:z\in D_t\}$.
\item[WHVF2.] For each $t\ge0$ the function $G(\cdot,t)$ is holomorphic in $D_t$.

\item[WHVF3.] For each compact set $K\subset\mathcal D$ there exists a non-negative function $k_K\in
L^d\big(\pr_{\mathbb R}(K),\Real\big)$, where $\pr_{\mathbb
R}(K):=\{t\ge0:~\exists\,z\in\Complex\quad(z,t)\in K\}$, such that
$$%
|G(z,t)|\le k_K(t),\quad\text{for all $(z,t)\in K$}.
$$%
\end{mylist}
\end{definition}

\begin{definition}
A weak holomorphic vector field $G$ over a canonical domain
system~$(D_t)$ is said to be {\it semicomplete}, if for any $s\ge0$
and any $z\in D_s$ the following initial value problem for the
Carath\'eodory ODE
$$
\dot w=G(w,t),\quad w|_{t=s}=z,
$$
has a solution defined for all $t\ge s$.
\end{definition}

In~\cite{SMP3} we proved the following statement establishing a
1-to-1 correspondence between evolution families and semicomplete
weak holomorphic vector fields.

\begin{result}[\protect{\cite[Theorem 5.1]{SMP3}}]\label{TH_EFsWHVF} The following two assertions hold:
\begin{mylist}
\item[(A)] For any $L^d$-evolution family~$(\varphi_{s,t})$ over the canonical
domain system~$(D_t)$ there exists an essentially unique semicomplete weak holomorphic vector field
$G:\mathcal D\to\Complex$ of order~$d$ and a null-set $N\subset [0,+\infty)$ such that for all
$s\ge0$ the following statements hold:
\end{mylist}
\begin{mylist}
\item[(i)] the mapping $[s,+\infty)\ni t\mapsto \varphi_{s,t}\in\Hol(D_s,\Complex)$
is locally absolutely continuous;
\item[(ii)] the mapping $[s,+\infty)\ni t\mapsto \varphi_{s,t}\in\Hol(D_s,\Complex)$ is
differentiable for all $t\in[s,+\infty)\setminus N$;
\item[(iii)] $d\varphi_{s,t}/dt=G(\cdot,t)\circ\varphi_{s,t}$ for all $t\in[s,+\infty)\setminus N$.
\end{mylist}\vskip3mm

\begin{mylist}
\item[(B)] For any semicomplete weak holomorphic vector field ${G:\mathcal D\to\Complex}$
of order~$d$ the formula $\varphi_{s,t}(z):=w^*_s(z,t)$, $t\ge s\ge0$, $z\in D_s$, where
$w_s^*(z,\cdot)$ is the unique non-extendable solution to the initial value problem
\begin{equation}\label{EQ_genGolKom}%
 \dot w=G(w,t),\quad w(s)=z,
\end{equation}%
defines an $L^d$-evolution family over the canonical domain system~$(D_t)$.
\end{mylist}
\end{result}
The exact meaning of the notions of absolute continuity and differentiability
of the mapping from (i)\,--\,(iii) in the above theorem is given
by~\cite[Definitions~2.7 and 2.8]{SMP3}.

The characterization of semicomplete weak holomorphic vector fields
we established in~\cite{SMP3} involves some notions from
Function Theory in the annulus. The analogue of the Schwartz kernel
$\mathcal K_0(z):=(1+z)/(1-z)$ for an annulus $\mathbb
A_{r}:=\{z:r<|z|<1\}$, $r\in(0,1)$, the so-called {\it Villat
kernel}, is defined by the following formula (see,
e.\,g.,~\cite{GoluzinM} or~\cite[\S\,V.1]{Aleksandrov}):
\begin{equation}\label{EQ_Villat_kernel}
\mathcal
K_r(z):=\lim\limits_{n\to+\infty}\sum\limits_{\nu=-n}^n\frac{1+r^{2\nu}z}{1-r^{2\nu}z}=\frac{1+z}{1-z}+\sum_{\nu=1}^{+\infty}
\left(\frac{1+r^{2\nu}z}{1-r^{2\nu}z}+\frac{1+z/r^{2\nu}}{1-z/r^{2\nu}}\right).
\end{equation}
It is known (see e.g.\,\cite[Theorem~2.2.10]{Vaitsiakhovich}) that for any function
$f\in\Hol(\mathbb A_r,\Complex)$ which is continuous in $\overline{\mathbb A_r}$,
\begin{multline}\label{EQ_theVillatFormula}
f(z)=\int_\UC\mathcal K_r(z\xi^{-1})\Re f(\xi)\,\frac{|d\xi|}{2\pi}+\int_\UC \big[\mathcal
K_r(r\xi/z)-1\big]\Re f(r\xi)\,\frac{|d\xi|}{2\pi}\\+i\int_{\UC}\Im
f(\rho\xi)\frac{|d\xi|}{2\pi}\quad \text{for all } z\in\mathbb A_r,~\rho\in[r,1].
\end{multline}

\begin{definition}
Let $r\in(0,1)$. By {\it the class $\ParClass_r$} we will mean the collection of all functions
$p\in\Hol(\mathbb A_r,\Complex)$ having the following integral representation
\begin{equation}\label{EQ_represV}
p(z)=\int_\UC\mathcal K_r(z/\xi)d\mu_1(\xi)+\int_\UC\big[1-\mathcal
K_r(r\xi/z)\big]d\mu_2(\xi),\quad z\in\mathbb A_r,
\end{equation}
where $\mu_1$ and $\mu_2$ are positive Borel measures on the unit circle~$\UC$ subject to the
condition~$\mu_1(\UC)+\mu_2(\UC)=1$.
\end{definition}

\begin{remark}\label{RM_mu1mu2-unique}
From the proof of~\cite[Theorem~1]{Zmorovich} it is evident that
given $p\in\ParClass_r$, the measures $\mu_1$ and $\mu_2$ in
representation~\eqref{EQ_represV} are unique. (See also the proof of
\cite[Lemma~5.11]{SMP3}.)
\end{remark}


\begin{result}[\protect{\cite[Theorem 5.6]{SMP3}}]\label{TH_semi-char-non-deg}
Let $d\in[1,+\infty]$ and let $\big(D_t\big)=\big(\mathbb
A_{r(t)}\big)$ be a non-degenerate canonical domain system of
order~$d$. Then a function $G:\mathcal D\to\Complex$, where
$\mathcal D:=\{(z,t):\,t\ge0,\,z\in D_t\}$, is a semicomplete weak
holomorphic vector field of order~$d$ if and only if there exist
functions $p:\mathcal D\to\Complex$ and $C:[0,+\infty)\to\Real$ such
that:
\begin{mylist}
\item[(i)] $G(w,t)=w\big[iC(t)+
r'(t)p(w,t)/r(t)\big]$ for a.e. $t\ge0$ and all $w\in D_t$;
\item[(ii)] for each $w\in D:=\cup_{t\ge0} D_t$ the function $p(w,\cdot)$ is measurable in
${E_w:=\{t\ge0:\,w\in D_t\}}$;
\item[(iii)]  for each $t\ge0$ the function $p(\cdot\,,t)$ belongs to the class $\ParClass_{r(t)}$;
\item[(iv)] $C\in L^d_{\mathrm{loc}}\big([0,+\infty),\Real\big)$.
\end{mylist}
\end{result}

\subsection{Main results}

In this paper we introduce a general notion of Loewner chain for
the doubly connected case and establish relationships between Loewner
chains and evolution families analogous to that in the Loewner
Theory for simply connected domains.

\begin{definition}\label{D_Loew_chain}
Let $d\in[1,+\infty]$ and $(D_t)$ be a canonical domain system of
order~$d$. A family $(f_t)_{t\ge0}$ of holomorphic functions
$f_t:D_t\to\Complex$ is called a {\it Loewner chain} of order $d$
(or in short an {\it $L^d$-Loewner chain}) over~$(D_t)$ if it
satisfies the following conditions:
\begin{enumerate}
\item[LC1.] each function $f_t:D_t\to\Complex$ is univalent,

\item[LC2.] $f_s(D_s)\subset f_t(D_t)$ whenever $0\leq s < t<+\infty,$

\item[LC3.] for any compact interval $I:=[S,T]\subset[0,+\infty)$ and any compact set $K\subset D_S$  there exists a
non-negative function $k_{K,I}\in L^{d}([S,T],\mathbb{R})$ such that
\[
|f_s(z)-f_t(z)|\leq\int_{s}^{t}k_{K,I}(\xi)d\xi
\]
for all $z\in K$ and all $(s,t)$ such that $S\leq s\leq t\leq T$.
\end{enumerate}
\end{definition}

The following theorem shows that every Loewner chain generates an
evolution family of the same order.

\begin{theorem}\label{TH_LC->EF}
Let $(f_{t})$ be a Loewner chain of order $d$ over a canonical
domain system $(D_{t})$ of order $d.$ If we define
\begin{equation}\label{EQ_LC->EF}
\varphi _{s,t}:=f_{t}^{-1}\circ f_{s},\qquad 0\leq s\leq t<\infty,
\end{equation}%
then $(\varphi _{s,t})$ is an evolution family of order $d$ over
$(D_{t}).$
\end{theorem}
This theorem is proved in Section~\ref{S_LC->EF}. As a consequence
we show, see Corollary~\ref{C_from_LC_to_VF}, that similarly to the
case of the unit disk, any Loewner chain over a canonical system of
annuli, satisfies a PDE driven by a semicomplete weak holomorphic
vector field.

In Section~\ref{S_from_EF_to_LC} we prove a converse of
Theorem~\ref{TH_LC->EF}, saying that for any evolution family
$(\varphi_{s,t})$ there exists a Loewner chain $(f_t)$ of the same
order such that $\eqref{EQ_LC->EF}$ holds and describe possible
conformal types of $\cup_{t\ge0} f_t(D_t)$. These results can be
formulated as follows. Denote by $I(\gamma)$ the index of the origin
w.r.t. a closed curve $\gamma\subset\C^*$. Similarly to the simply
connected case~\cite{SMP}, we will say that a Loewner chain
$(f_{t})$ over $(D_t)$ is {\it associated with} an evolution
family~$(\varphi_{s,t})$ over the same canonical domain system
if~\eqref{EQ_LC->EF} holds.

\begin{theorem}\label{Th_from_EF_to_LC}
Let $\left((D_t),(\varphi_{s,t})\right)$, where $D_{t}:=\mathbb
A_{r(t)}$ for all $t\ge0$, be an evolution family of order
$d\in[1,+\infty]$. Let $r_\infty:=\lim_{t\to+\infty} r(t)$. Then
there exists a Loewner chain $(f_t)$ of order $d$ over $(D_t)$ such
that
\begin{enumerate}
       \item $f_s=f_t\circ\varphi_{s,t}$ for all $0\leq s\leq t<+\infty$, i.e. $(f_t)$ is associated with $(\varphi_{s,t})$;
       \item $I(f_t\circ\gamma)=I(\gamma)$ for any closed curve $\gamma\subset D_{t}$ and any $t\ge0$;
       \item If $0<r_\infty<1$, then $\cup_{t\in [0,+\infty)} f_t(D_t)=\mathbb A_{r_\infty}$;
       \item If $r_\infty =0$, then $\cup_{t\in [0,+\infty)} f_t(D_t)$ is either $\D^*$, $\C\setminus \overline{\D}$, or $\C^*$.
     \end{enumerate}
If $(g_t)$ is another Loewner chain over $(D_t)$ associated with
$(\varphi_{s,t})$, then there is a biholomorphism $F:\cup_{t\in
[0,+\infty)} g_t(D_t)\to \cup_{t\in [0,+\infty)} f_t(D_t)$ such that
$f_t=F\circ g_t$ for all $t\ge0$.
\end{theorem}

In general, a Loewner chain associated with a given evolution family is not
unique. We call a Loewner chain $(f_t)$ to be {\it standard} if it satisfies
conditions (2)\,--\,(4) from Theorem~\ref{Th_from_EF_to_LC}. It follows from
this theorem that the standard Loewner chain $(f_t)$ associated with a given
evolution family, is defined uniquely up to a rotation (and scaling if
$\cup_{t\in [0,+\infty)} f_t(D_t)=\C^*$). Furthermore, combining
Theorems~\ref{TH_LC->EF} and~\ref{Th_from_EF_to_LC} one can easily conclude
that for any Loewner chain~$(g_t)$ of order~$d$ over a canonical domain
system~$(D_t)$ there exists a biholomorphism $F:\cup_{t\in [0,+\infty)}
g_t(D_t)\to L[(g_t)]$, where $L[(g_t)]$ is either $\D^*$, $\C\setminus
\overline{\D}$, $\C^*$, or $\mathbb A_{\rho}$ for some $\rho>0$, such that the
formula $f_t=F\circ g_t$, $t\ge0$, defines a standard Loewner chain of
order~$d$ over the canonical domain system~$(D_t)$. This motivates the
following definition.
\begin{definition}\label{D_conftyp}
Let $(g_t)$ be a Loewner chain of order $d$ over a canonical domain
system~$(D_t)=(\mathbb A_{r(t)})$. Let $r_\infty:=\lim_{t\to+\infty}
r(t)$.  We say that
\begin{itemize}
\item $(g_t)$ is {\it of (conformal) type I}, if $L[(g_t)]=\mathbb A_{\rho}$ for some $\rho>0$;
\item $(g_t)$ is {\it of (conformal) type II}, if $L[(g_t)]=\D^*$;
\item $(g_t)$ is {\it of (conformal) type III}, if $L[(g_t)]=\C\setminus \overline{\D}$;
\item $(g_t)$ is {\it of (conformal) type IV}, if $L[(g_t)]=\C^*$.
\end{itemize}
By the {(conformal) type} of an evolution family $(\varphi_{s,t})$
we mean the conformal type of any Loewner chain associated
with~$(\varphi_{s,t})$.
\end{definition}

Suppressing the notation we will also write $L[(\varphi_{s,t})]$
meaning $L[(g_t)]$, where $(g_t)$ is any Loewner chain associated
with $(\varphi_{s,t})$. We call this domain the {\it Loewner range}
of $(\varphi_{s,t})$.

\begin{remark}
The domains $\D^*$ and $\C\setminus \overline{\D}$ are conformally
equivalent. However, the conformal types II and III can be
distinguished because of the condition~(2) from
Theorem~\ref{Th_from_EF_to_LC}. Indeed, $I(G\circ\gamma)=-I(\gamma)$
for any biholomorphism $G:\D^*\to \C\setminus \overline{\D}$ and any
closed curve~$\gamma\subset \D^*$.
\end{remark}

The following statements, proved in Section~\ref{S_conf_type_EF}
characterize the conformal type of a Loewner chain via the
properties of the corresponding evolution family. Consider an
evolution family $(\varphi_{s,t})$ over a canonical domain system
$(D_t)=(\mathbb A_{r(t)})$, where $r(t)>0$ for all $t\ge0$. Denote
$r_\infty:=\lim_{t\to+\infty}r(t)$. Further for each $s\ge0$ and
$t\ge s$, define
$\tilde\varphi_{s,t}(z):=r(t)/\varphi_{s,t}(r(s)/z)$. By
\cite[Example 6.3]{SMP3}, $(\tilde \varphi_{s,t})$ is also an
evolution family over~$(D_t)$. Note that at least one of the
families $(\varphi_{0,t})$ and $(\tilde\varphi_{0,t})$ converges
to~$0$ as $t\to+\infty$ provided $r_\infty=0$ (see Lemma~\ref{LM_ML}).

\begin{theorem}\label{TH_types_nondeg}
Let $\big((D_t),(\varphi_{s,t})\big)$ be a non-degenerate evolution
family. In the above notation, the following statements hold:
\begin{itemize}
  \item[(i)] the evolution family $(\varphi_{s,t})$ is of type I if and only
  if $r_\infty>0$;
  \item[(ii)] the evolution family $(\varphi_{s,t})$ is of type II if and only if $r_\infty=0$ and $\varphi_{0,t}$  does not converge
  to 0  as  $t\to+\infty$;
  \item[(iii)] the evolution family $(\varphi_{s,t})$ is of type III if and only if $r_\infty=0$ and $\tilde{\varphi}_{0,t}$  does not converge
  to 0  as  $t\to+\infty$;
  \item[(iv)] the evolution family $(\varphi_{s,t})$ is of type IV if and only if $r_\infty=0$  and both $\varphi_{0,t}\to0$ and $\tilde{\varphi}_{0,t}\to0$
  as $t\to+\infty$.
\end{itemize}
\end{theorem}
\noindent In the mixed-type or degenerate case the situation is
simpler. Namely, we prove following
\begin{proposition}\label{TH_types_mixed}
Let $(\varphi_{s,t})$ be an evolution family over a canonical domain
system $(D_t)=(\mathbb A_{r(t)})$. Assume that $r(T)=0$ for some
$T\in[0,+\infty)$, i.e. $(D_t)$ is of mixed-type or degenerate. Then
$(\varphi_{s,t})$ is of type IV if $\varphi_{0,t}\to0$  as
$t\to+\infty$, and of type II otherwise.
\end{proposition}

Further new results of the present paper are as follows. As we
mentioned in Section~\ref{SS_prel} each evolution family is
generated by a weak holomorphic vector field. So it is possible to
study the limit behavior of an evolution family using the
corresponding vector fields. In this way we obtain a necessary and
sufficient condition for a non-degenerate evolution family
$(\varphi_{s,t})$ to satisfy the condition~$\varphi_{s,t}\to0$ as
$t\to+\infty$, see Theorem~\ref{TH_EF0} in Section~\ref{S_convzero}.

In~\cite{SMP3} we obtained an explicit characterization of
semicomplete weak holomorphic vector fields over {\it
non-degenerate} canonical domain systems. As an application of
general Loewner Theory in the unit disk we also obtained
in~\cite{SMP3} an analogous result for {\it degenerate} canonical
domain systems. In this paper we include Section~\ref{S_mixed-type}
devoted to obtaining a characterization of semicomplete weak
holomorphic vector fields over canonical domain systems of {\it
mixed type}.

Finally, in the short Section~\ref{S_conf-type-VF} we combine the
above results to obtain the conformal classification of Loewner
chains, in doubly connected setting, via the corresponding weak
holomorphic vector fields.

\section{From Loewner chains to evolution families}\label{S_LC->EF}

In this Section we prove Theorem~\ref{TH_LC->EF}. The proof is based on the
following lemmas. In what follows, using the notation $[a,b]$, we allow  $a$ to
be equal to~$b$. In such case $[a,b]$ means the singleton $\{a\}$.

\begin{lemma}\label{LM_LipInv}
Let $\big(D_t\big)=\big(\mathbb A_{r(t)}\big)$ be a canonical domain system and
$(f_t)$ a Loewner chain over $(D_t)$. Then for any compact set
$K\subset\mathcal D:=\{(z,t):~t\ge0,\,z\in D_t\}$ there exists $M=M(K)>0$ such
that
\begin{equation*}
|\zeta-z|\le M|f_t(\zeta)-f_t(z)|\quad\text{whenever $(z,t),(\zeta,t)\in K$.}
\end{equation*}
\end{lemma}
\begin{proof}
Assume the contrary. Then there exist sequences $(\zeta_n)$, $(z_n)$
and $(t_n)$ such that $(\zeta_n,t_n),(z_n,t_n)\in K$ and
\begin{equation}\label{EQ_assume1}
|\zeta_n-z_n|>n|f_{t_{n}}(\zeta_n)-f_{t_{n}}(z_n)|
\end{equation}%
for every $n\in \mathbb{N}$.

By the compactness of $K$ we may assume that the sequences
$(\zeta_n)$, $(z_n)$, and $(t_n)$ converge to some $\zeta_0$, $z_0$,
and $t_0$, respectively. Clearly, $\zeta_0,z_0\in D_{t_0}$, because
$(\zeta_0,t_0)$ and $(z_0,t_0)$ belong to $K$. Using continuity of
$[0,+\infty)\ni t\mapsto r(t)$ we therefore conclude that there
exist $n_1\in\Natural$ and $\tau_1\in[0,t_0]$ such that
$K_1:=\{\zeta_n,\,z_n:\,n>n_1\}\cup\{\zeta_0,z_0\}$ is a compact
subset of $D_{\tau_1}$ and $t_n\ge\tau_1$ for all $n>n_1$.

Now we note that by LC3, $f_{t}\to f_{t_0}$ uniformly on compact
subsets of $U:=D_{\tau_1}$ as $t\to t_0$, $t\ge\tau_1$. In
particular it follows that $f_{t_n}(\zeta_n)\to f_{t_0}(\zeta_0)$
and $f_{t_n}(z_n)\to f_{t_0}(z_0)$ as $n\to+\infty$, $n>n_1$.
Moreover, any compact set $B\subset f_{t_0}(U)$ is also contained in
$f_t(U)$ if $t\ge\tau_1$ and $|t-t_0|$ is small enough\footnote{This
statement can be easily obtained in the same way as in the proof of
the Carath\'eodory kernel convergence theorem (see, e.g., \cite[\S
II.5, Theorem\,1]{Goluzin}) or \cite[p.\,29,
Theorem\,1.8]{Pommerenke}).}. Hence we conclude that there exist
$n_2>n_1$, $n_2\in\Natural$, and $\tau_2\in[\tau_1,t_0]$ such that
$K_2:=\{f_{t_n}(\zeta_n),\,f_{t_n}(z_n):\,n>n_2\}\cup\{f_{t_0}(\zeta_0),f_{t_0}(z_0)\}$
is a compact subset of $W:=f_{\tau_2}(U)$ and $t_n\ge\tau_2$ for all
$n>n_2$.

According to the definition of a Loewner chain over a doubly
connected canonical domain system, the functions
$g_n:=\big(f_{t_n}^{-1}\big)|_W$ are well-defined and holomorphic in
$W$ for all $n>n_2$. Moreover, $g_n(W)\subset\UD$ for any $n>n_2$.
Hence the family $\mathcal F:=\{g_n:\,n>n_2\}$ is normal in~$W$ and
its closure in $\Hol(W,\Complex)$ is compact. Therefore, there
exists $M'=M'(K_2,\mathcal F)>0$ such that
$$
|g_n(w_2)-g_n(w_1)|\le M'|w_2-w_1|\quad\text{for any $w_1,w_2\in
K_2$ and any $n>n_2$}.
$$
Choosing $w_1:=f_{t_n}(z_n)$ and $w_2:=f_{t_n}(\zeta_n)$ we see that
the above inequality contradicts~\eqref{EQ_assume1} for large
$n\in\Natural$. This contradiction completes the proof.
\end{proof}

\begin{lemma}\label{LM_Comp}
Under the conditions of Lemma~\ref{LM_LipInv}, let $K$ be a compact subset of
$\mathcal U:=\{(w,t):\,t\ge0,\,w\in f_t(D_t)\}$. Then $\hat
K:=\big\{\big(f_t^{-1}(w),t\big):\,(w,t)\in K\big\}$ is a compact subset of
$\mathcal D$.
\end{lemma}
\begin{proof}
Consider an arbitrary sequence $\big((z_n,t_n)\big)\subset \hat K$. We need to
prove that it has a subsequence converging to a point in $K$. To this end write
$w_n:=f_{t_n}(z_n)$ for any $n\in\Natural$. According to the compactness of
$K$, passing if necessary to a subsequence we may assume that $(t_n)$ converges
to some $t_0$ and $(w_n)$ converges to some $w_0\in f_{t_0}(D_{t_0})$.

It is sufficient to show that $z_n\to z_0:=f_{t_0}^{-1}(w_0)$ as $n\to+\infty$.
Indeed, in this case $\big((z_n,t_n)\big)$ converges to ${(z_0,t_0)\in\hat K}$.

To prove that $z_n\to z_0$ as $n\to+\infty$ we fix $\varepsilon>0$
small enough. Denote $B_\varepsilon:=\{z:|z-z_0|\le \varepsilon\}$,
$C_\varepsilon:=\{z:|z-z_0|= \varepsilon\}$. From the continuity of
$t\mapsto r(t)$ and from the fact that $w_n\to w_0$ as
$n\to+\infty$, it follows that there exists $n_0\in\Natural$ such
that $B_\varepsilon\subset D_{t_n}$ and $w_n\in
f_{t_0}(B_\varepsilon\setminus C_\varepsilon)$ for all $n>n_0$. Let
$S:=\min\{t_n:\,n>n_0\}\cup\{t_0\}$,
$T:=\max\{t_n:\,n>n_0\}\cup\{t_0\}$. Then from LC3  it follows that
$f_{t_n}\to f_{t_0}$ uniformly on $B_\varepsilon$ as $n\to+\infty$,
$n>n_0$. Hence there exists $n_1>n_0$ such that
$$
\max_{z\in C_\varepsilon}|f_{t_n}(z)-f_{t_0}(z)|<\min_{z\in
C_\varepsilon}|f_{t_0}(z)-w_0|-|w_n-w_0|
$$
for all $n>n_1$. Note that
$|f_{t_0}(z)-w_0|-|w_n-w_0|\le|f_{t_0}(z)-w_n|$.

Recall that by the construction, the open disk
$B_\varepsilon\setminus C_\varepsilon$ contains the unique solution
of $f_{t_0}(z)-w_n=0$ provided $n>n_0$. Then by the Rouche theorem
for the functions $f_{t_n}-w_n$ and $f_{t_0}-w_n$, for each $n>n_1$,
the unique solution to $f_{t_n}(z)-w_n=0$, which is $z=z_n$, belongs
to $B_\varepsilon\setminus C_\varepsilon$. Therefore,
$|z_n-z_0|<\varepsilon$ for all $n>n_1$. Since $\varepsilon>0$ was
chosen arbitrarily, this shows that $z_n\to z_0$ as $n\to+\infty$
and hence the proof is complete.
\end{proof}

\begin{lemma}\label{LM_LC3}
Under the conditions of Lemma~\ref{LM_LipInv}, let $K$ be a compact subset of
$\mathcal D$ and $$E:=\big[\min_{(z,t)\in K} t,\max_{(z,t)\in K} t\big].$$ Then
there exists a non-negative function $k_{K}\in L^d(E,\Real)$ such that
$$
|f_t(z)-f_u(z)|\le \int_{u}^{t}k_K(\xi)\,d\xi
$$
for any $z\in\Complex$ and any $u,t\ge0$ satisfying $(z,u), (z,t)\in K$ and
$u\le t$.
\end{lemma}
\begin{proof}
Since $t\mapsto r(t)$ is continuous,  for any $(\zeta,s)\in K$ there
exists $\rho>0$ and $\varepsilon>0$ such that $\{z:\,|z-\zeta|\le
\rho\}\subset D_S$, where $S:=\max\{0,s-\varepsilon\}$. It follows
that
$$K_{(\zeta,s)}:=\{z:\,|z-\zeta|\le \rho\}\times\big[
\max\{0,s-\varepsilon\},T\big]\subset \mathcal D,\quad
T:=1+\max_{(z,t)\in K} t,$$ for any $(\zeta,s)\in K$.

Let $U_{(\zeta,s)}$ stand for the interior of $K_{(\zeta,s)}$ w.r.t.
$\Complex\times[0,+\infty)$. Then
$$
K\subset\bigcup_{(\zeta,s)\in K}U_{(\zeta,s)}.
$$
Therefore, by the compactness of $K$, there exist finite sequences
$\zeta_1,\ldots,\zeta_n\in\Complex$, $S_1,\ldots, S_n\in[0,T]$, and
$\rho_1,\ldots,\rho_n>0$ such that
$$
 K\subset\bigcup_{j=1}^n K_j\times I_j,
$$
where $K_j:=\{z:|z-\zeta_j|\le \rho_j\}\subset D_{S_j}$, and
$I_j:=[S_j,T]\subset[0,+\infty)$ for all $j=1,\ldots,n$.

Then by LC3, there exist non-negative functions $k_j:=k_{K_j,I_j}\in
L^d(I_j,\Real)$ such that for each $j=1,\ldots,n$ and any $z\in K_j$,
$$
|f_t(z)-f_u(z)|\le \int_{u}^t k_j(\xi)\,d\xi
$$
whenever $u,t\in I_j$ and $u\le t$.

Finally, we notice that by construction, for arbitrary $z\in\Complex$ and
$u,t\ge0$ satisfying $(z,u), (z,t)\in K$,  there exists $j=1,\ldots,n$ such
that $z\in K_j$ and $u,t\in I_j$. Thus the statement of the lemma holds with
$$
k_K:=\sum_{j=1}^n \chi_{I_j} k_j,
$$
where $\chi_{I_j}$ stands for the characteristic function of the set $I_j$.
This completes the proof.
\end{proof}

\begin{proof}[{\bf Proof of Theorem~\ref{TH_LC->EF}}] It is straightforward to
check that for any $s\ge0$ and $t\ge s$, formula~\eqref{EQ_LC->EF} defines  a
holomorphic mapping $\varphi_{s,t}:D_{s}\to D_{t}$ and that the family
$(\varphi_{s,t})$ satisfies conditions EF1 and EF2.

To prove EF3, fix $[S,T]\subset [0,+\infty)$ and $z\in D_S$. From LC2 and LC3
it follows that, the set $K:=\{(f_s(z),t):\,S\le s\le t\le T\}$ is a compact
subset of $\mathcal U$ (as a continuous image of a compact set). Then by
Lemma~\ref{LM_Comp},
$$
\hat K:=\big\{\big(f_t^{-1}(f_s(z)),t\big):\,S\le s\le t\le
T\big\}=\big\{\big(\varphi_{s,t}(z),t\big):\,S\le s\le t\le T\big\}
$$
is a compact subset of~$\mathcal D$.

Consider the continuous mapping $G:\mathcal D\times[0,1]\to\mathcal
D$ defined by
$$G:\big((\zeta,u),\lambda\big)\mapsto \big(\zeta,(1-\lambda)u+\lambda T\big).$$ Since
$\hat K\times[0,1]$ is compact, the set
$$
K_0:=G(\hat K)=\big\{\big(\varphi_{s,u}(z),t\big):\,S\le s\le u\le
t\le T\big\}
$$
is again a compact subset of $\mathcal D$. Clearly, $\hat K\subset
K_0$.

Apply now Lemma~\ref{LM_LipInv} with $K_0$ substituted for $K$. Then
there exists $M>0$ such that
\begin{equation*}
\big|\varphi_{s,t}(z)-\varphi_{s,u}(z)\big|\le M
\big|f_t\big(\varphi_{s,t}(z)\big)-f_t\big(\varphi_{s,u}(z)\big)\big|=
M\big|f_u\big(\varphi_{s,u}(z)\big)-f_t\big(\varphi_{s,u}(z)\big)\big|
\end{equation*}
whenever $S\le s\le u\le t\le T$. Since for any such $s$, $u$, and
$t$, the points $\big(\varphi_{s,u}(z),u\big)$ and
$\big(\varphi_{s,u}(z),t\big)$ belong to $K_0$, applying
Lemma~\ref{LM_LC3} with $K_0$ substituted for~$K$ completes the
proof.
\end{proof}

We conclude this section with a corollary relating Loewner chains
with PDEs.

\begin{corollary}\label{C_from_LC_to_VF}
Let $d\in[1,+\infty]$. Let $(D_t)$ be a canonical domain system of
order~$d$ and $(f_{t})$ a Loewner chain of order $d\in[1,+\infty]$
over $(D_t)$. Then the following statements hold:

\begin{enumerate}
\item[(i)] There exists a null-set $N\subset [0,+\infty)$ (not depending on $z$)
such that for every $s\in [0,+\infty)\setminus N$ the function
\begin{equation*}
z\in D_s\mapsto \frac{\partial f_{s}(z)}{\partial s}
:=\lim_{h\rightarrow 0}\frac{f_{s+h}(z)-f_{s}(z)}{h}\in \C
\end{equation*}
is a well-defined holomorphic function on $D_s$.

\item[(ii)] There exists an essentially unique weak holomorphic vector field $G$ of order $d$ over~$(D_t)$
such that for a.e. $s\in [0,+\infty)$,
\begin{equation}\label{EQ_GK-PDE}
\frac{\partial f_{s}(z)}{\partial s}=-G(z,s)f_{s}^{\prime
}(z)\quad~\text{for all $z\in D_s$}.
\end{equation}
The evolution family $(\varphi_{s,t})$ of the Loewner chain~$(f_t)$
solves for every fixed $s\ge0$ and $z\in D_s$ the ODE
$$\frac{d\varphi_{s,t}(z)}{dt}=G\big(\varphi_{s,t}(z),t\big),\quad \text{a.e. $t\ge s$}.$$
\end{enumerate}
\end{corollary}
Essential uniqueness means here that any two vector fields
satisfying~\eqref{EQ_GK-PDE} can differ only for values of~$s$
forming a null-set on the real line.

The proof of Corollary~\ref{C_from_LC_to_VF} is very similar to that
of \cite[Theorem~4.1(1)]{SMP}, so we omit it. We call the vector
field $G$ in the second statement the {\it vector field associated
with} the Loewner chain~$(f_t)$.

\section{From evolution families to Loewner chains}\label{S_from_EF_to_LC} This section is devoted to the proof
of~Theorem~\ref{Th_from_EF_to_LC} establishing the existence, in the
doubly connected setting, of a standard Loewner chain of order~$d$
associated with a given evolution family of order $d$.

Important role in our discussion is played by the class
$\classM(r_1,r_2)$ of all functions $\psi\in\Hol(\mathbb
A_{r_1},\mathbb A_{r_2})$, $1>r_1\ge r_2\ge 0$, such that
$I(\psi\circ\gamma)=I(\gamma)$ for every closed curve
$\gamma\subset\mathbb A_{r_1}$.

We will make use of the following two lemmas.

\begin{lemma}[{\cite[Lemma 4.7]{SMP3}}]\label{LM_evol-fam-classM}
Suppose $\big((D_t),(\varphi_{s,t})\big)$ is an evolution family of
order $d\in[1,+\infty]$. Let $s\ge 0$. Then the following statements
are true:
\begin{itemize}
\item[(i)] for each $z\in D_s$ the function $t\mapsto \varphi_{s,t}(z)$ belongs to
$AC^d\big([s,+\infty),\Complex\big)$;
\item[(ii)] the mapping $t\mapsto\varphi_{s,t}\in \Hol(\mathbb A_{r(s)},\mathbb D^*)$ is continuous in $[s,+\infty)$;
\item[(iii)] $\varphi_{s,t}\in \classM\big(r(s),r(t)\big)$ for any $t\ge s$;
\item[(iv)] $\varphi_{s,t}$ is univalent in $D_s$ for any $t\ge s$.
\end{itemize}
\end{lemma}

\begin{lemma}\label{LM_LC3a}
Let $\big((D_t),(\varphi_{s,t})\big)$ be an evolution family of
order $d\in[1,+\infty]$. Let $(f_t)_{t\ge0}$ be a family of
univalent functions $f_t:D_t\to\Complex$. If
$f_s=f_t\circ\varphi_{s,t}$ for any $s\ge0$ and any $t\ge s$, then
$(f_t)$ is a Loewner chain of order~$d$ associated with
$\big((D_t),(\varphi_{s,t})\big)$.
\end{lemma}
\begin{proof}
We follow ideas of
the 4th step in the proof of \cite[Theorem 4.5]{ABHK}. Conditions
LC1 and LC2 in Definition~\ref{D_Loew_chain} follow easily from the
condition of the lemma: $f_s(D_s)=f_t(\varphi_{s,t}(D_s))\subset
f_t(D_t)$ for any $s\ge0$ and any $t\ge s$. So we only need to prove
LC3.

First of all, we note that by \cite[Theorem 5.1(A)]{SMP3} there
exists a semicomplete weak holomorphic vector field
$G:\{(z,t):\,t\ge0,~z\in D_t\}\to\Complex$ such that for any $s\ge0$
and $z\in D_s$, the function $[s,+\infty)\ni t\mapsto
w^*_s(z,\cdot)$ solves the initial value problem $\dot w=G(w,t)$,
$w(s)=z$. It follows (see, e.g., \cite[Theorem 2.3(v)]{SMP3}) that
$\varphi_{s,t}(z)$ is jointly continuous in $z$, $s$, and $t$. Fix
arbitrary $I:=[S,T]\subset[0,+\infty)$ and a compact set $K\subset
D_S$. Then
$$\hat K:=\bigcup_{S\le s\le t\le T}\varphi_{s,t}(K)\quad\text{and}
\quad\tilde K:=\big\{\big(\varphi_{s,t}(z),t\big):\,z\in K,~ S\le
s\le t\le T\big\}$$ are compact sets in $D_T$ and in $\{(z,t):\,t\in
I,~z\in D_t\}$, respectively.

Since $f_T$ is holomorphic in $D_T$, there exists $C=C(\hat K,T)$
such that
\begin{equation}\label{EQ_C1}
|f_T(w)-f_T(z)|\le C|w-z|\quad\text{for any $z,w\in \hat K$.}
\end{equation}

We claim that there exists $C'=C'(\tilde K)>0$ such that
\begin{equation}\label{EQ_C2}
|\varphi_{t,T}(w)-\varphi_{t,T}(z)|\le C'|w-z|\quad\text{whenever
$(z,t),(w,t)\in\tilde K$.}
\end{equation}
Let us assume the contrary. Then there exist sequences $(z_n)$,
$(w_n)$ and $(t_n)$ such that $(z_n,t_n),(w_n,t_n)\in\tilde K$ and
\begin{equation}\label{EQ_assume}
|\varphi_{t_n,T}(w_n)-\varphi_{t_n,T}(z_n)|>n|w_n-z_n|\quad
\text{for all $n\in\Natural$}.
\end{equation}
By the compactness of $\tilde K$ we may assume that the sequences
$(z_n)$, $(w_n)$ and $(t_n)$ converge to some $z_0$, $w_0$ and
$t_0$, respectively. Clearly, $t_0\in I$ and $z_0,w_0\in D_{t_0}$.
Moreover, the left-hand side of~\eqref{EQ_assume} is bounded and
consequently $w_0=z_0$. Since, by the definition of a canonical
domain system, $D_t=\mathbb A_{r(t)}$, where $r:[0,+\infty)\to[0,1)$
is continuous, there exists $n_0\in\Natural$ and $\tau\in I$ such
that $t_n\in[\tau,T]$ for all $n>n_0$ and $X:=\{z_n:n>
n_0\}\cup\{w_n:n> n_0\}\cup\{z_0\}$ is a compact subset of $D_\tau$.
Since $\varphi_{t,T}(D_\tau)\subset\UD$ for any $t\in[\tau,T]$, the
family $\{\varphi_{t,T}:t\in[\tau,T]\}$ is normal in $D_\tau$ and
consequently there exists $M:=M(X,\tau,T)>0$ such that
$$
|\varphi_{t,T}(w)-\varphi_{t,T}(z)|\le M |w-z|\quad\text{for any
$z,w\in X$ and $t\in[\tau,T]$.}
$$
The latter contradicts~\eqref{EQ_assume}. This proves~\eqref{EQ_C2}.

Further, by \cite[Proposition 4.5]{SMP3} there exists non-negative
function $k_{K,I}\in L^d\big(I,R\big)$ such that
\begin{equation}\label{EQ_C3}
|\varphi_{s,t}(z)-\varphi_{s,u}(z)|\le\int_u^t k_{K,I}(\xi)d\xi
\end{equation}
for any $z\in K$ and all $s,u,t\in I$ satisfying $s\le u\le t$.

Thus from~\eqref{EQ_C1}, \eqref{EQ_C2} and \eqref{EQ_C3} we deduce
that for any $s\in I$, $t\in[s,T]$, and $z\in K$,
\begin{eqnarray*}
|f_t(z)-f_s(z)|&=&|f_T(\varphi_{s,T}(z))-f_T(\varphi_{t,T}(z))|\\
& \le &
C|\varphi_{s,T}(z)-\varphi_{t,T}(z)|=C|\varphi_{t,T}(\varphi_{s,t}(z))-\varphi_{t,T}(z)|\vphantom{\int_s^t
k_{K,I}(\xi)d\xi}\\
&\le &
CC'|\varphi_{s,t}(z)-z|=CC'|\varphi_{s,t}(z)-\varphi_{s,s}(z)|\le
CC'\int_s^t k_{K,I}(\xi)d\xi.
\end{eqnarray*}
This completes the proof.
\end{proof}

Now we recall some basic properties of the module of a doubly
connected domain. Given any path-connected topological space $X$, we
denote by $\Pi_1(X)$ its fundamental group. By the {\it base point}
of a closed curve $\gamma:[0,1]\to X$ we mean the point
$\gamma(0)=\gamma(1)$. Let $G\subset\Complex$ be a doubly connected
domain. A closed curve $\gamma\subset G$ is {\it homotopically
nontrivial in $G$} if $\gamma$ is not homotopic in $G$ to its base
point, i.e. if the equivalence class $[\gamma]_{\Pi(G)}$ is not the
neutral element of $\Pi_1(G)$. It is known (see, e.g., \cite[Chapter
1.D, Example 3]{Ahlfors}) that for every doubly connected domain
$G\subset\Complex$ there exists a quantity $M(G)\in (0,+\infty]$
called the {\it module\footnote{Denote by $\Gamma$ the set of all closed rectifiable curves $\gamma\subset G$ that are homotopically nontrivial in~$G$. By definition, the {\it module} of $G$ is the reciprocal of the extremal length of $\Gamma$.} of $G$} having the following properties:

\begin{itemize}
        \item[M1.] The module is invariant under conformal mappings, i.e. if $G_2=f(G_1)$ and $f$ is a conformal mapping of $G_1$, then
        $M(G_2)=M(G_1)$.
        \item[M2.] If $G_1\subset G_2$ and any closed curve in $G_1$ homotopically nontrivial in $G_1$ is also
        homotopically nontrivial in $G_2$,
        then $M(G_1)\leq M(G_2)$.
        \item[M3.] We have $M(\D^*)=M(\C^*)=+\infty$.
        \item[M4.] If $0<r_1 <r_2$, then $M(\mathbb A_{r_1,r_2})=\frac{1}{2\pi}\log\left({r_2}/{r_1}\right)$.
\end{itemize}

Finally, we will make use of the following remark without explicit
reference.

\begin{remark}
Let $\gamma_j$, $j=1,2$, be closed curves in $G:=\mathbb
A_{r_1,r_2}$, $0\le r_1<r_2\le+\infty$. The curve $\gamma_1$ is
homotopic to $\gamma_2$ if and only if $I(\gamma_1)=I(\gamma_2)$.
\end{remark}

\begin{proof}[\bf Proof of Theorem~\ref{Th_from_EF_to_LC}.]

We divide the proof into several steps.

\noindent{\bf Step 1.} {\it There exists a Riemann surface $N$ and a
family of mappings $(g_t:D_t\to N)$ such that
\begin{itemize}
  \item[(i)] $g_t$ is univalent for any $t\ge0$;
  \item[(ii)] $g_s(D_s)\subset g_t(D_t)$ whenever $0\leq s<t<+\infty$;
  \item[(iii)] $N=\cup_{t\ge0}g_t(D_t)$;
  \item[(iv)] $g_s=g_t\circ \varphi_{s,t}$ whenever $0\leq s<t<+\infty$.
\end{itemize}}

This statement can be easily established if one follows the proof of
\cite[Theorem 4.5]{ABHK}, bearing in mind that in our case the
domains of the functions $\varphi_{s,t}$ depend on $s$. Therefore we
omit here the proof.

\noindent{\bf Step 2.} {\it The surface $N$ is doubly connected}.

\noindent{\underline{Step 2a.}} {\it If, for some $s\ge0$, a closed
curve $\gamma\subset N_s:=g_s(D_s)$ is homotopically nontrivial in
$N_s$, then it is also homotopically nontrivial in $N$ and in $N_t$
for all $t\ge s$.}

Indeed, for any $t\ge s$, a closed curve $\gamma\subset N_s$ is
homotopically non-trivial in $N_t$ if and only if
$I(g_t^{-1}\circ\gamma)\neq0$. According to the property (iv) above
and Lemma~\ref{LM_evol-fam-classM}\,(iii),
\begin{equation}\label{EQ_Igamma}
I(g_t^{-1}\circ\gamma)=I(\varphi_{s,t}\circ g_s^{-1}\circ\gamma)=
I(g_s^{-1}\circ\gamma).
\end{equation}
Therefore, if $\gamma:[0,1]\to N_s$, $\gamma(0)=\gamma(1)$, is
homotopically nontrivial in $N_s$, then it is also homotopically
nontrivial in $N_t$ provided $t\ge s$. Suppose that, at the same
time, $\gamma$ is homotopically trivial in $N$. Then there exists a
homotopy $H:[0,1]\times[0,1]\to N$ such that $H(0,\cdot)=\gamma$ and
$H(1,\cdot)\equiv \const$. In view of (ii) and (iii), by the
compactness of $H([0,1]\times[0,1])$ we have
$H([0,1]\times[0,1])\subset N_t$ for all $t\ge0$ large enough and
hence $\gamma$ is homotopically trivial in $N_t$ for at least one
$t\ge s$. This contradicts the statement we have just proved. Thus
$\gamma$ must be also homotopically nontrivial in $N$.

\noindent{\underline{Step 2b.}} {\it The fundamental group
$\Pi_1(N)$ is not trivial.}

Fix $z_0\in D_0$. The fundamental groups $\Pi_1(N_t)$, $t\ge0$, and
$\Pi_1(N)$ can be realized as groups of equivalence classes
$[\gamma]_{\Pi_1(N_t)}$ (respectively, $[\gamma]_{\Pi_1(N)}$) of
closed curves with the base point at $g_0(z_0)$. Note that since
each surface $N_t:=g_t(D_t)$ is doubly connected, the fundamental
group $\Pi_1(N_t)$ is isomorphic to $\Z$ for any $t\geq 0$.

Consider a closed curve $\alpha:[0,1]\to D_{0}$ with the base point
at $z_0$ such that $I(\alpha)=1$. Then the equivalence class
$[\gamma_0]_{\Pi_1(N_0)}$ of $\gamma_0:=g_0\circ \alpha$ generates
the fundamental group $\Pi_1(N_0)$. By Step 2a, it follows, in
particular, that $\gamma_0$ is not homotopically trivial in $N$, so
the fundamental group $\Pi_1(N)$ is not trivial.

\noindent{\underline{Step 2c.}} {\it The fundamental group
$\Pi_1(N)$ is generated by one element.}

By \eqref{EQ_Igamma} with $\gamma:=\gamma_{0}$ and $s:=0$, we have
$I(g_t^{-1}\circ\gamma_0)=1$ for all $t\ge0$. Thus the equivalence
class $[\gamma_0]_{\Pi_1(N_t)}$ of $\gamma_0$ also generates the
fundamental group $\Pi_1(N_t)$.

We claim that $[\gamma_0]_{\Pi_1(N)}$ generates  the fundamental
group $\Pi_1(N)$. Indeed, take any closed curve $\gamma:[0,1]\to N$
with the base point at $g_0(z_0)$. Combining (ii), (iii) and the
compactness of $\gamma([0,1])$, we see that there exists $t\ge0$
such that $\gamma([0,1])\subset N_t$. Therefore, there must exists
$n\in\Z$ such that
$[\gamma]_{\Pi_1(N_t)}=\big([\gamma_0]_{\Pi_1(N_t)}\big)^n$. But any
homotopy in $N_t$ is also a homotopy in $N$ and this implies that
$[\gamma]_{\Pi_1(N)}$ belongs to the subgroup generated by
$[\gamma_0]_{\Pi_1(N)}$. This proves our claim.

Thus we have showed that $\Pi_1(N)$ is non-trivial and generated by
one element. In particular, it is Abelian. According to
\cite[Theorem 1.129]{Abate}, ${\Pi_1(N)}$ is isomorphic either $\Z$
or $\Z\times\Z$. Since it is generated by one element, ${\Pi_1(N)}$
has to be isomorphic to $\Z$. That is, $N$ is doubly connected. By
\cite[Corollary 1.1.30]{Abate}, we conclude that there is a
biholomorphism $\mathcal H$ from $N$ onto either $\C^*$, $\D^*$, or
an annulus $\mathbb A_r$ for some $0<r<1$.

Write $h_t=\mathcal H\circ g_t$. According to (i)\,--\,(iv),
\begin{itemize}
  \item[(i')] $h_t$ is univalent for all $t\ge0$;
  \item[(ii')] $h_s(D_s)\subset h_t(D_t)$ whenever $0\leq s<t<+\infty$;
  \item[(iii')] $\Omega:=\cup_{t\ge0}h_t(D_t)$ is either $\C^*$, $\D^*$, or an annulus $\mathbb A_r$ for some $0<r<1$;
  \item[(iv')] $h_s=h_t\circ \varphi_{s,t}$ whenever $0\leq
  s<t<+\infty$.
\end{itemize}

\noindent{\bf Step 3.} {\it Let $r_\infty:=\lim_{r\to+\infty}r(t)$.
If $r_\infty=0$, then $\Omega \in\{\D^*,\C^*\}$. Otherwise,
$\Omega=\mathbb A_{r_\infty}$.}

If $\Omega =\mathbb A_r$ for some $0<r<1$, set $a:=r$. Otherwise,
put $a:=0$. To simplify the exposition, we will assume as usual that
$\frac{1}{0}=+\infty$ and $\log(+\infty)=+\infty$.

Take $\varepsilon\in(0,(1-a)/2)$. Then $\overline{\mathbb
A_{a+\varepsilon,1-\varepsilon}}\subset\Omega$. By (ii'), (iii') and
the compactness of $\overline{\mathbb
A_{a+\varepsilon,1-\varepsilon}}$, there is $t_0\ge0$ such that
$\mathbb G_\varepsilon:=\mathbb
A_{a+\varepsilon,1-\varepsilon}\subset h_t(D_t)$ for all $t>t_0$.
Moreover, it is clear that any closed curve $\gamma\subset \mathbb
G_\varepsilon$ homotopically nontrivial in $\mathbb G_\varepsilon$
is also homotopically nontrivial in $\Omega$ and hence must be
homotopically nontrivial in $h_t(D_t)\subset \Omega$. Therefore,
from M1\,--\,M4 we get
\begin{equation*}
   \frac{1}{2\pi}\log\left(\frac{1-\varepsilon}{a+\varepsilon}\right)= M(\mathbb G_\varepsilon)
   \leq M(h_t(D_t))=M(D_t)=\frac{1}{2\pi}\log \left(\frac{1}{r(t)}\right).
\end{equation*}
Thus $(1-\varepsilon)/(a+\varepsilon)\leq 1/r(t)$. Passing to the
limit as $t\to+\infty$ and then letting $\varepsilon\to+0$ we get
$r_\infty \leq a$.

On the other hand, by Step 2a, any homotopically nontrivial closed
curve in $N_s$ is also homotopically nontrivial in $N$. By conformal
equivalence, this statement can be translated to the domains
$h_s(D_s)$ and $\Omega$. Hence, using again M1\,--\,M4, we may
conclude that
\begin{equation*}
    \frac{1}{2\pi}\log \left(\frac{1}{r(s)}\right)=M(D_s)=M(h_s(D_s))\leq M(\Omega)= \frac{1}{2\pi}\log
    \left(\frac{1}{a}\right).
\end{equation*}
Passing to the limit as $s\to+\infty$, we obtain the inequality
$r_\infty \geq a$. Therefore $r_\infty = a$.

This means that if $r_\infty=0$, then $M(\Omega)=\infty$ and $\Omega
\in\{\D^*,\C^*\}$, while for $r_\infty>0$ we have $\Omega =\mathbb
A_{r_\infty}$.

\noindent{\bf Step 4.} {\it There is $\kappa\in\{-1,1\}$ such that
$I(h_t\circ\gamma)=\kappa I(\gamma)$ for any $t\ge0$ and any closed
curve $\gamma\in D_t$.}

Fix $z_0\in D_0$. First of all we note that given $t\ge0$, any
closed curve $\gamma\subset D_t$ is homotopic in $D_t$ to some
closed curve $\tilde \gamma\subset D_0$ with the base point at
$z_0$. In particular this means that
\begin{equation}\label{EQ_It0}
I(\gamma)=I(\tilde\gamma).
\end{equation}
Moreover, by Lemma~\ref{LM_evol-fam-classM}\,(iii), $I(\tilde
\gamma)=I(\varphi_{0,t}\circ\tilde\gamma)$.  This means that
$\varphi_{0,t}\circ\tilde\gamma$ and $\tilde\gamma$ are homotopic in
$D_t$. Therefore $h_t\circ\gamma$ is homotopic to
$h_t\circ\varphi_{0,t}\circ\tilde\gamma=h_0\circ\tilde\gamma$ in
$\Omega$. Hence
\begin{equation}\label{EQ_Ith}
I(h_0\circ\tilde\gamma)=I(h_t\circ\gamma).
\end{equation}
From \eqref{EQ_It0} and~\eqref{EQ_Ith} it follows that in the proof of Step~4
we may fix $t:=0$ and assume that $\gamma$ has the base point at~$z_0$.

Now we claim that the mapping $g_0$ establishes the isomorphism
$G^{g_0}$ between $\Pi_1(D_0)$ and $\Pi_1(N)$ that takes the
equivalence class $[\gamma]_{\Pi_1(D_0)}$ of each closed curve
$\gamma\subset D_0$ with the base point at $z_0$ to the equivalence
class $[g_0\circ\gamma]_{\Pi_1(N)}$ of~$g_0\circ\gamma$. Indeed,
$G^{g_0}$ is a well-defined  group homomorphism. Furthermore,
according to the argument of Step~2, both fundamental groups are
isomorphic to $\mathbb Z$ and the generator $[\alpha]_{\Pi_1(D_0)}$
of $\Pi_1(D_0)$ is mapped by $G^{g_0}$ to the generator
$[g_0\circ\alpha]_{\Pi_1(N)}$ of $\Pi_1(N)$. Thus $G^{g_0}$ is an
isomorphism.

Further, the biholomorphism $\mathcal H:N\to\Omega$ defines in the
canonical way the isomorphism $G^{\mathcal
H}:\Pi_1(N)\to\Pi_1(\Omega)$.

Notice now that for the domains $D_0$ and $\Omega$ there exist a
canonical isomorphisms of their fundamental groups onto $\mathbb Z$,
$G_D:\Pi_1(D)\to\mathbb Z$, $D\in\{D_0,\Omega\}$, defined in the
following way: $G_D$ takes each $[\gamma]_{\Pi_1(D)}$ to
$I(\gamma)$.

Now consider the isomorphism $G_{\mathbb Z}:=G_{\Omega}\circ
G^{\mathcal H} \circ G^{g_0}\circ (G_{D_0})^{-1}:\mathbb Z\to\mathbb Z$.
The only two isomorphisms of $\mathbb Z=(\mathbb Z,+)$ onto itself
are the identity $G_{\mathbb Z}=\id_{\mathbb Z}$ and $G_{\mathbb
Z}:\mathbb Z\ni n\mapsto -n$. In the former case we have $I(h_0\circ
\gamma)=I(\gamma)$ for any closed curve $\gamma\subset D_0$ with the
base point at $z_0$, while in the latter case we have $I(h_0\circ
\gamma)=-I(\gamma)$ for all such $\gamma$'s.

Set $f_t:=h_t$ for all $t\ge0$ if $\kappa=1$, $f_t:=r_\infty/h_t$
for all $t\ge0$ if $\kappa=-1$, $r_\infty>0$, and $f_t:=1/h_t$ for
all $t\ge0$ if $\kappa=-1$, $r_\infty=0$.

\noindent{\bf Step 5.} {\it $(f_t)$ is a standard Loewner chain of
order $d$ over $(D_t)$ associated with $(\varphi_{s,t})$.}

From (i') and (iv') it follows that $(f_t)$ satisfies the condition of
Lemma~\ref{LM_LC3a}. Hence $(f_t)$ is an $L^d$-Loewner chain associated with
$\big((D_t),(\varphi_{s,t})\big)$.

The fact that $(f_t)$ is a standard Loewner chain follows from
(iii'), the definition of $f_t$, and Steps 3 and~4.

\noindent{\bf Step 6.} {\it If $(g_t)$ is another Loewner chain
associated with $\big((D_t),(\varphi_{s,t})\big)$, then there is a
biholomorphism $F:\cup_{t\ge0}f_t(D_t) \to \cup_{t\ge 0}g_t(D_t)$
such that $g_t=F\circ f_t$ for all $t\ge0$.}

The proof of this step is similar to an argument from the proof of
\cite[Theorem 4.9]{ABHK}, so we omit it.

Since the statement of the theorem is the combination of Step 5 and
Step 6, the proof is now finished.
\end{proof}

\section{Conformal types of Loewner chains via evolution
families}\label{S_conf_type_EF}

This section is devoted to the classification of Loewner chains in
terms of the limit behavior of their evolution families. We will
prove Theorem~\ref{TH_types_nondeg} and
Proposition~\ref{TH_types_mixed} giving such a classification.

The proofs are based on following lemmas.

It is known that given a Jordan curve $\gamma\subset\C$,  there
exists $\kappa\in \{1,-1\}$ such that the index of $w$ w.r.t.
$\gamma$, $\textrm{ind}(\gamma, w)\in \{0,\kappa\}$ for all $w\in
\C\setminus \gamma$. As usual, we denote by
$\textrm{int}(\gamma):=\{w\in \C\setminus
\gamma:\textrm{ind}(\gamma, w) \neq 0  \}$ and
$\textrm{out}(\gamma):=\{w\in \C\setminus
\gamma:\textrm{ind}(\gamma, w) = 0  \}$.

\begin{lemma} \label{Argument_Principle}
Let $f:\mathbb A_r \to \C^*$ be a univalent function such that
$I(f\circ\gamma)=I(\gamma)$ for any closed curve
$\gamma\subset\mathbb A_{r}$. Then  $f(z)\in
\mathrm{out}\big(f(C(0,R))\big)$ whenever $r<R<|z|<1$.
\end{lemma}
\begin{proof}
Consider an arbitrary $z_0$ satisfying $R<|z_0|<1$.  Let
$w_0:=f(z_0)$. We have to prove that $\textrm{ind}(f\circ C^-,
w_0)=0$, where  $C^-$ is the circle $C(0,R):=\{z: |z|=R\}$ oriented
clockwise.

Fix $\tilde R\in\big(|z_0|,1\big)$. By $C^+$ we denote the circle
$C(0,\tilde{R})$ oriented counter-clockwise.

Since the equation $f(z)-w_0=0$ has exactly one zero $z=z_0$ in the
annulus $\mathbb A(R,\tilde{R}):=\{z:R<|z|<\tilde R\}$, by the
argument principle we have
\begin{equation}\label{Argument_Principle_1}
1=\frac{1}{2\pi i}\int_{\partial \mathbb
A(R,\tilde{R})}\frac{f'(z)}{f(z)-w_0}dz= \frac{1}{2\pi
i}\int_{f\circ C^+}\frac{dw}{w-w_0}+\frac{1}{2\pi i}\int_{f\circ
C^-}\frac{dw}{w-w_0}.
\end{equation}
By hypothesis, $1=I(C^+)=I(f\circ C^+)$. Therefore, 
\begin{equation}\label{Argument_Principle_2}
\frac{1}{2\pi i}\int_{f\circ
C^+}\frac{dw}{w-w_0}=\textrm{ind}(f\circ C^+, w_0)\in \{0,1\}.
\end{equation}
Analogously, 
\begin{equation}\label{Argument_Principle_3}
\frac{1}{2\pi i}\int_{f\circ
C^-}\frac{dw}{w-w_0}=\textrm{ind}(f\circ C^-, w_0)\in \{0,-1\}.
\end{equation}
Clearly, equations (\ref{Argument_Principle_1}),
(\ref{Argument_Principle_2}), and (\ref{Argument_Principle_3}) show
that $\textrm{ind}(f\circ C^-, w_0)=0$. This completes the proof of
the lemma.
\end{proof}
\begin{lemma}
\label{bound_on_annulus} Let $f:\mathbb A_r \to \D^*$, $r\in[0,1)$,
be a univalent function such that $I(f\circ\gamma)=I(\gamma)$ for
any closed curve $\gamma\subset\mathbb A_{r}$. Then
\begin{equation}\label{EQ_f-est}
    |f(z)|\leq \frac{\pi}{\sqrt{2\log(1/|z|)}}
\end{equation}
for all $z\in \mathbb A_r$.
\end{lemma}
\begin{proof} In this proof we will use again the notion and properties of module of a doubly connected domain, see
Section~\ref{S_from_EF_to_LC}.

Fix $R\in (r,1)$. Take $z_0$ such that $|z_0|=R$ and
$|f(z_0)|=N:=\max\{|f(z)|:\ |z|=R\}$. Write $w_0:=f(z_0)$. Consider
an arbitrary closed rectifiable curve $\gamma\subset f(\mathbb A_R)$
with $I(\gamma)\neq 0$. Denote by $L_1$ the ray $-w_0[0,+\infty)$.
Since $I(\gamma)\neq 0$, we have $\gamma\cap L_1\neq\emptyset$.

In a similar way we may conclude that $\gamma\cap L_2\neq\emptyset$,
where $L_2$ stands for the ray $w_0[1,+\infty)$. Indeed, the union
of $E:=f(C(0,R))\cup\textrm{int}\big(f(C(0,R))\big)\cup
L_2\cup\{\infty\}$ contains a Jordan arc connecting the origin with
$\infty$. Hence $I(\gamma)\neq 0$ implies that $\gamma\cap
E\neq\emptyset$. Moreover, by Lemma~\ref{Argument_Principle},
$\gamma\subset\textrm{out}(f(C(0,R))$. Therefore, $\gamma\cap
L_2\neq\emptyset$.

Since $\gamma$ is closed and $\gamma\cap L_j\neq\emptyset$, $j=1,2$,
it follows that the Euclidean length of $\gamma$ is at least
$2|w_0|=2N$.

Define $\rho_0(z):=1$ for $z\in \D$ and $\rho_0(z):=0$ for $z\notin
\D$. Obviously, $\rho_0\in L^2(\C)$. We denote by
$\textrm{len}_{\rho_0}(\gamma)$ the length of $\gamma$ with respect
to the metric $\rho_0(z)|dz|^2$.
 Then, by  the very definition of the module of a doubly connected domain (see, e.g., \cite[Section
I.D, Example~3]{Ahlfors}), we have
\begin{equation*}
    \frac{(2N)^2}{\pi} \leq \frac{[\inf_\gamma \textrm{len}_{\rho_0}(\gamma)]^2}{||\rho_0||_{L^2(\C)}^2} \leq
    \frac{1}{M(f(\mathbb A_R))}=\frac{1}{M(\mathbb A_R)}=\frac{2\pi}{\log (1/R)},
\end{equation*}
where the infimum is taken over all closed rectifiable curves
$\gamma\subset f(\mathbb A _R)$ with $I(\gamma)\neq 0$. Hence
\begin{equation*}
    N\leq \frac{\pi}{\sqrt{2}\sqrt{\log(1/R)}}.
\end{equation*}
This finishes the proof.
\end{proof}

\begin{lemma}\label{LM_ML}
Let $\big(D_t\big)=\big(\mathbb{A}_{r(t)}\big)$ be a canonical domain system of order
$d\in[1,+\infty]$ and $(\varphi_{s,t})$ an evolution family of the same order $d\in[1,+\infty]$
over $\big(D_t\big)$. Then the following statements hold:
\begin{mylist}
\item[(i)] For any $s\ge0$ and any sequence $(t_n)\subset[s,+\infty)$ there exists a subsequence of
$\psi_n:=\varphi_{s,t_n}$ that converges uniformly on compacta in~$D_s$ either to a constant or to
a univalent holomorphic function $\psi:D_s\to\UD^*$. In the latter case, $\psi\in\mathbb M\big(r(s),0\big)$.

\item[(ii)] If there exist $s_0\ge0$ and a sequence $(t_n)\subset[s_0,+\infty)$
such that $t_n\to+\infty$ and $\varphi_{s_0,t_n}$ converges to a
constant as $n\to+\infty$, then for all $s\ge0$, $\varphi_{s,t}\to0$
uniformly on compacta in~$D_s$ as $t\to+\infty$.

\item[(iii)] Either $\rho_{z,s}(t):=|\varphi_{s,t}(z)|\to0$ for all
$s\ge0$ and all $z\in D_s$, or there exists a positive function
$\delta:\mathbb D\to(0,+\infty)$ such that $\delta(z,s)<\rho_{z,s}(t)<1-\delta(z,s)$ whenever $0\le s\le t$ and $z\in D_s$.
\end{mylist}
\end{lemma}
\begin{proof}
Statement (i) follows from the fact that the sequence $(\psi_n)$
forms a normal family in $D_s$ and from Hurwitz's theorem. Indeed,
$\psi_{n}(D_s)\subset\UD$ for all $n\in\Natural$. Hence $(\psi_n)$
is normal in $D_s$. Let $(\psi_{n_k})$ be a subsequence converging
uniformly on compacta in $D_s$ to a function $\psi$. All the
functions $\psi_n$ are univalent in $D_s$. Hence, by the Hurwitz
theorem, $\psi$ is either a constant in $\overline \UD$, or $\psi$
is univalent in $D_s$ and $\psi(D_s)\subset\UD^*$. Clearly, given a
closed curve $\sigma:[0,1]\to\C^*$, there exists $\varepsilon>0$
such that $I(\tilde\sigma)=I(\sigma)$ for any closed curve $\tilde
\sigma:[0,1]\to\C^*$ satisfying the inequality
$|\tilde\sigma(t)-\sigma(t)|<\varepsilon$ for all $t\in[0,1]$.
Taking $\sigma:=\psi\circ\gamma$ and $\tilde
\sigma:=\psi_{n_k}\circ\gamma$, where $\gamma$ is an arbitrary
closed curve in $D_s$, we therefore conclude that $\psi\in\mathbb
M\big(r(s),0\big)$, unless $\psi\equiv0$.

The above argument proves~(i) and shows that if $\psi=\const$, then
$\psi\equiv0$. Therefore, to prove~(ii) we may assume that
$\varphi_{s_0,t_n}\to0$ as $n\to+\infty$. Recall for any $s\ge0$ and
any~$t\ge s$,
\begin{equation}\label{EQ_recall}
\varphi_{0,t}=\varphi_{s,t}\circ\varphi_{0,s}.
\end{equation}
Therefore, $\varphi_{0,t_n}\to0$ in $D_0$ as $n\to+\infty$. The
convergence is uniform on compacta because the family
$(\varphi_{0,t})_{t\ge0}$ is normal in $D_0$. Fix now $s\ge0$.
Taking into account that $\varphi_{0,s}$ is non-constant and using
again~\eqref{EQ_recall} and the normality of $(\varphi_{s,t})_{t\ge
s}$ in $D_s$, we conclude now that
\begin{equation}\label{EQ_s-t_n}
\varphi_{s,t_n}\to 0\text{~~uniformly on compact subsets of~$D_s$ as
$n\to+\infty$.}
\end{equation}
Take any $n\in\Natural$ such that $t_n\ge s$ and let $t\ge t_n$. By
Lemma~\ref{LM_evol-fam-classM}, the function $f:=\varphi_{t_n,t}$
satisfies the condition of Lemma~\ref{bound_on_annulus} with
$r:=r(t_n)$. Note that
$\varphi_{s,t}=\varphi_{t_n,t}\circ\varphi_{s,t_n}$. Hence
\eqref{EQ_f-est} and \eqref{EQ_s-t_n} imply that $\varphi_{s,t}\to0$
as $t\to+\infty$ uniformly on compacta in~$D_s$. This proves~(ii).

It remains to prove~(iii). To this end assume
that~$\rho_{z_0,s_0}(t_n)\to1$ or $\rho_{z_0,s_0}(t_n)\to0$ for some
$s_0\ge0$, some $z_0\in D_s$, and some sequence
$(t_n)\in[s_0,+\infty)$. Since $0<\rho_{z,s_0}(t)<1$ for all
$t\in[s_0,+\infty)$ and the function $\rho_{z_0,s_0}$ is continuous
by Lemma~\ref{LM_evol-fam-classM}(i), we have that $t_n\to+\infty$
as $n\to+\infty$. Moreover,
$\varphi_{s_0,t_n}(D_{s_0})\subset\mathbb D^*$ for all $n\ge0$.
Hence passing if necessary to a subsequence, we may conclude that
$\varphi_{s_0,t_n}$ converges to a constant as $n\to+\infty$. But
then by~(ii), for any $s\ge0$, $\varphi_{s,t}\to0$ in $D_s$ as
$t\to+\infty$, i.e. $\rho_{z,s}(t)\to0$ as $t\to+\infty$ for any
$s\ge0$ and any $z\in D_s$. This proves (iii).
\end{proof}

Now we can prove Theorem~\ref{TH_types_nondeg}.

\begin{proof}[{\bf Proof of Theorem~\ref{TH_types_nondeg}.}]
Statement (i) of the theorem is already proved: it is equivalent the
statement of Step 3 in the proof of Theorem~\ref{Th_from_EF_to_LC}.

Now we assume $r_\infty=0$.

First we prove (ii). So suppose that $\varphi_{0,t}$ does not
converge to $0$ as $t\to+\infty$. We have to prove that the Loewner
range $L[(\varphi_{s,t})]$ of $(\varphi_{s,t})$ is $\D^*$. According
to Lemma~\ref{LM_ML}, there exists a sequence
$(t_n)\subset[0,+\infty)$ diverging to~$+\infty$ such that
$(\varphi_{0,t_n})$ converges to some univalent function
$\varphi_{0,\infty}$ uniformly on compacta in $D_0$. For given
$s\ge0$ and all $n\in\Natural$ large enough, by EF2 we have
$\varphi_{0,t_n}=\varphi_{s,t_n}\circ\varphi_{0,s}$, with
$\varphi_{0,s}$ being univalent in $D_0$ by
Lemma~\ref{LM_evol-fam-classM}(iv). Using again Lemma~\ref{LM_ML}
and taking into account the normality of $(\varphi_{s,t})_{t\ge s}$
in $D_s$ we may conclude that $(\varphi_{s,t_n})$ also converges to
some univalent function $\varphi_{s,\infty}$ in $D_s$ and that
\begin{equation}\label{EQ_0infty}
\varphi_{0,\infty}=\varphi_{s,\infty}\circ\varphi_{0,s}.
\end{equation}
By EF2, $\varphi_{0,s}=\varphi_{s,u}\circ\varphi_{0,u}$ for any
$u\in[0,s]$. In combination with~\eqref{EQ_0infty} this gives
$\varphi_{s,\infty}\circ\varphi_{s,u}\circ\varphi_{0,u}=\varphi_{u,\infty}\circ\varphi_{0,u}$.
Since $\varphi_{0,u}$ is not constant, by the Identity Theorem for
holomorphic functions we get
$\varphi_{s,\infty}\circ\varphi_{s,u}=\varphi_{u,\infty}$. This
holds for any $s,u\ge0$ with $u\le s$. Then by Lemma~\ref{LM_LC3a},
$(\varphi_{t,\infty})_{t\ge0}$ is a Loewner chain over $(D_t)$
associated with $(\varphi_{s,t})$. By
Theorem~\ref{Th_from_EF_to_LC}, there exists a standard Loewner
chain $(f_t)$ associated with $(\varphi_{s,t})$ and a biholomorphism
$F:L[(\varphi_{s,t})]\to\Omega:=\cup_{t\ge0}\varphi_{t,\infty}(D_t)\subset
\D^*$ such that $\varphi_{t,\infty}=F\circ f_t$ for all $t\ge0$.

We claim that for any closed curve $\gamma\subset
L[(\varphi_{s,t})]$,
\begin{equation}\label{EQ_Fgamma}
I(F\circ\gamma)=I(\gamma).
\end{equation}
Indeed, fix such a curve $\gamma$. By the compactness of $\gamma$,
there exists $t\ge0$ such that $\gamma\subset f_t(D_t)$ and hence
$\gamma=f_t\circ\gamma_t$ for some closed curve $\gamma_t\subset
D_t$. On the one hand, by the definition of a standard Loewner chain
$I(\gamma)=I(f_t\circ\gamma_t)=I(\gamma_t)$. On the other hand,
$\varphi_{t,\infty}\in\mathbb M(r(t),0)$ by Lemma~\ref{LM_ML}(i) and
hence $I(\varphi_{t,\infty}\circ\gamma_t)=I(\gamma_t).$ Recall that
$\varphi_{t,\infty}=F\circ f_t$. Now~\eqref{EQ_Fgamma} follows
easily.

Further, we note that $L[(\varphi_{s,t})]\neq\C^*$ because $\Omega$
is a bounded domain. Moreover, from statement (i) we know that
$L[(\varphi_{s,t})]\neq\mathbb A_r$ for any $r\ge0$. Hence
$L[(\varphi_{s,t})]\in\{\D^*,\Complex\setminus\overline\D\}$ and it
remains to show that
$L[(\varphi_{s,t})]\neq\Complex\setminus\overline\D$. Assume on the
contrary that $L[(\varphi_{s,t})]\neq\Complex\setminus\overline\D$.
The function $F(1/z)$ is holomorphic and bounded in~$\D^*$. Hence it
has a removable singularity at $z=0$. Let $G$ be its holomorphic
extension to $\UD$. Apply the Argument Principle to this function on
the circle $\gamma:=C(0,1/2)$ oriented counterclockwise, so that
$I(\gamma)=1$. Then on the one hand, $I(G\circ\gamma)\ge0$ because
$G$ has no poles in $\UD$. But on the other hand
$I(G\circ\gamma)=-1$ by~\eqref{EQ_Fgamma}. This contradiction proves
that ${L[(\varphi_{s,t})]=\UD^*}$.

To complete the proof of statement (ii) we have to show that if
$\varphi_{0,t}\to0$, then ${L[(\varphi_{s,t})]\neq \D^*}$. Assume the
contrary and let $(f_t)$ stand again for a standard Loewner chain
associated with $(\varphi_{s,t})$. Fix any $z\in D_0$. By
Lemma~\ref{bound_on_annulus} applied for $f_t:\mathbb A_{r(t)}\to
\D^*$ with $t\ge0$, we get
$$
|f_0(z)|=|f_t(\varphi_{0,t}(z))|\le\frac{\pi}{\sqrt{2\log\big(1/|\varphi_{0,t}(z)|\big)}}.
$$
The left-hand side tends to zero as $t\to+\infty$. Therefore,
$f_0(z)=0$ for all $z\in D_0$. This contradiction shows that
${L[(\varphi_{s,t})]\neq \D^*}$. The proof of (ii) is now finished.

To prove (iii), we only need to apply statement~(ii) to
$(\tilde\varphi_{s,t})$ instead of $(\varphi_{s,t})$ and note that
if $(f_t)$ is a standard Loewner chain associated
with~$(\varphi_{s,t})$ and $r_\infty=0$, then by
Lemma~\ref{LM_LC3a}, $(\tilde f_t)$ is a standard Loewner chain
associated with~$(\tilde\varphi_{s,t})$, where $\tilde
f_t(z):=1/f_t(r(t)/z)$, $t\ge0$, $z\in D_t$.

Finally, statement (iv) holds by the exclusion principle:
$L[(\varphi_{s,t})]=\Complex^*$ if and only if
$L[(\varphi_{s,t})]\not
\in\{\D^*,\Complex\setminus\overline\D,\mathbb A_r:\,r\in(0,1)\}$.
The proof is now complete.
\end{proof}

At the end of the section we prove Proposition~\ref{TH_types_mixed}
giving conformal characterization of a Loewner chain via its
evolution family in the degenerate and mixed-type cases.

\begin{proof}[{\bf Proof of Proposition~\ref{TH_types_mixed}}] Let $(f_t)$ be a standard Loewner chain associated with $\big((D_t),\varphi_{s,t}\big)$.
Notice that $(\varphi_{T+s,T+t})$ is an evolution family over
$(D_{T+t})$, whose standard Loewner chain is $(f_{T+t})$. Moreover,
it is evident that $L[(f_{T+t})]=L[(f_t)]$.

Therefore, we may assume that $T=0$ and
$\big((D_t),\varphi_{s,t}\big)$ is of degenerate type. According
to~\cite[Proposition 5.15]{SMP3}, the functions defined by
$\phi_{s,t}(z):=\varphi_{s,t}(z)$ for $z\in\UD^*$,
$\phi_{s,t}(0)=0$, $0\le s\le t$, form in this case an evolution
family in the unit disk~$\UD$. By~\cite[Theorem 1.6]{SMP} there
exist a Loewner chain~$(g_t)$ in the unit disk~$\UD$ associated with
$(\phi_{s,t})$ such that $g_t(0)=g_t(\phi_{0,t}(0))=g_0(0)=0$ and
$\Omega:=\cup_{t\ge0}g_t(\UD)$ is either a Euclidian disk centered
at the origin or the whole complex plane~$\C$. Moreover, according
to the same theorem, $\Omega=\Complex$ if and only if
$\phi'_{0,t}(0)\to 0$ as $t\to+\infty$. Clearly, the latter
condition is equivalent to the requirement that $\varphi_{0,t}\to0$
as $t\to+\infty$. Finally, we notice that (up to scaling in case
$\Omega\neq\Complex$) the family $(g_t|_{\UD^*})$ is a standard
Loewner chain associated with~$(\varphi_{s,t})$. This finishes the
proof, since $\cup_{t\ge0}g_t(\UD^*)=\Omega\setminus\{0\}$.
\end{proof}

\section{Non-degenerate evolution families: convergence to
zero}\label{S_convzero}

For any $r\in[0,1)$ and any $f\in\Hol(\mathbb A_{r},\Complex)$ we denote by $\mathcal N(f)$ the
free term in the Laurent development of $f$:
$$
\mathcal N(f):=\int_{\UC}f(\rho\xi)\frac{|d\xi|}{2\pi},\quad \rho\in(r,1).
$$

\begin{theorem}\label{TH_EF0}
Let $\big(D_t\big)=\big(\mathbb{A}_{r(t)}\big)$ be a non-degenerate canonical domain system of
order $d\in[1,+\infty]$ and $(\varphi_{s,t})$ an evolution family of the same order
$d\in[1,+\infty]$ over $\big(D_t\big)$. Suppose that $r_\infty:=\lim_{t\to+\infty}r(t)=0$.  Then
the following two statements are equivalent:
\begin{mylist}
\item[(A)] For any $s\ge0$, $\varphi_{s,t}\to0$ uniformly on compacta in~$D_s$ as $t\to+\infty$.
\item[(B)] The weak holomorphic vector field $G:\mathcal D\to\Complex$ associated with~$(\varphi_{s,t})$ satisfies the condition
\begin{equation}\label{EQ_int}
\int\limits_0^{+\infty}\Re\mathcal N\big(D_t\ni w\mapsto
G(w,t)/w\big)\,dt=-\infty.
\end{equation}
\end{mylist}
\end{theorem}
\begin{remark}\label{RM_int}
According to Theorem~\ref{TH_semi-char-non-deg}, the weak
holomorphic vector field $G$ in the above theorem has the following
representation
\begin{equation*}
G(w,t)=w\left[iC(t)+\frac{r'(t)}{r(t)}p(w,t)\right],\quad \text{for a.e. $t\ge0$ and all $w\in D_t$},
\end{equation*}
where $C\in L^d_{\mathrm{loc}}\big([0,+\infty),\Real\big)$, and the function $p$ is measurable
in~$t$ and belongs, as a function of $w$, to the class $\mathcal V_{r(t)}$ for every fixed $t\ge0$.
Denote by $\mu_1^t$ and $\mu_2^t$ the measures in representation~\eqref{EQ_represV} for
$p(\cdot,t)$. Then condition~\eqref{EQ_int} takes the following form:
\begin{equation}\label{EQ_intM}
-\int_{0}^{+\infty}r'(t)\frac{\mu_1^t(\UC)}{r(t)}dt=+\infty,
\end{equation}
while the negation of~\eqref{EQ_int} is equivalent to the convergence of the above integral,
because the integrand is non-positive for a.e. $t\ge0$.
\end{remark}

\begin{proof}[{\bf Proof of Theorem~\ref{TH_EF0}}]
Fix any $z\in D_0$. Denote $w(t):=\varphi_{0,t}(z)$ and $\rho(t):=|w(t)|$ for all $t\ge0$.

In this proof we use the notation introduced in Remark~\ref{RM_int}.
Using this remark, from the equation $\dot w=G(w,t)$ we get
\begin{equation}\label{EQ_Drho}
\frac{d\rho(t)}{dt}=r'(t)\frac{\rho(t)}{r(t)}\Re p(w(t),t).
\end{equation}

Denote $\nu(t):=\mu_1^t(\UC)$. Note that $0\le\nu(t)\le1$ for all
$t\ge0$. From representation~\eqref{EQ_represV} and properties of
the Villat kernel~$\mathcal K_{r(t)}$ (see, e.g., \cite[Remark
5.2]{SMP3}) it follows that
\begin{multline}\label{EQ_estReP}
\nu(t)\mathcal K_{r(t)}\big(-\rho(t)\big)+(1-\nu(t))\big[1-\mathcal
K_{r(t)}\big(r(t)/\rho(t)\big)\big]\le\Re p(w(t),t)\le\\
\nu(t)\mathcal K_{r(t)}(\rho(t))+(1-\nu(t))\big[1-\mathcal
K_{r(t)}\big(-r(t)/\rho(t)\big)\big].
\end{multline}

Using the Laurent development of the Villat kernel, we get
\begin{equation}\label{EQ_est1}
\mathcal K_r(x)-1=2\sum_{k=1}^{+\infty}\frac{x^k-(r^2/x)^k}{1-r^{2k}}\le 2\sum_{k=1}^{+\infty}\frac{x^k}{1-r^{2k}} \le \frac{2}{1-r^2}\frac{x}{1-x},\quad 0<r<x<1,
\end{equation}
while from~\eqref{EQ_Villat_kernel} it follows that
\begin{equation}\label{EQ_est2}
\mathcal K_r(-x)\ge \mathcal K_0(-x)=\frac{1-x}{1+x},\quad 0<r<x<1.
\end{equation}

Let us first prove that (B) implies (A). Assume that statement (A) does not hold. Then by
Lemma~\ref{LM_ML}\,(iii), we have $1-\delta>\rho(t)>\delta$ for some positive constant $\delta$ and
all $t\ge0$. Recall that $r(t)\to0$ as $t\to+\infty$. Since the functions $t\mapsto r(t)$ and
$t\mapsto \rho(t)$ are continuous and satisfy inequality $r(t)<\rho(t)$ for all $t\ge0$, we can
conclude that there exists $\delta_1>0$ such that $\rho(t)>r(t)+\delta_1$. Then taking into account
that $t\mapsto \rho(t)$ is locally absolutely continuous in $[0,+\infty)$ and that $r'(t)\le0$ for
a.e. $t\ge0$, from~\eqref{EQ_Drho}\,--\,\eqref{EQ_est2} we get that for all $T>0$,
\begin{multline}\label{EQ_nuzhno}
\rho(0)-\rho(T)\ge -\int_0^T
r'(t)\frac{\rho(t)}{r(t)}\left[\nu(t)\frac{1-\rho(t)}{1+\rho(t)}-(1-\nu(t))\frac{2r(t)}{1-r(t)^2}
\frac1{\rho(t)-r(t)}\right]dt\\\ge -\frac{\delta^2}{2}\int_0^T
r'(t)\frac{\nu(t)}{r(t)}dt+\frac2{\delta_1}\int_0^T \frac{r'(t)}{1-r(t)^2}dt.
\end{multline}
The left-hand side and the second term in the right-hand side of the above inequality are bounded. Hence the integral
$$
-\int_0^T r'(t)\frac{\nu(t)}{r(t)}dt
$$ is bounded from above. With the help of Remark~\ref{RM_int} it follows that statement (B) fails to be true. Thus,
(B)$\Rightarrow$(A).

It remains to prove that (A)$\Rightarrow$(B). Assume on the contrary that (B) does not hold, while
(A) is true. Then on the one hand, the integral in~\eqref{EQ_intM} converges, but on the other
hand, $\rho(t)\to0$ as $t\to+\infty$. To obtain a contradiction we need another estimate for
$\mathcal K_r(x)$. Using again the Laurent development of the Villat kernel, we obtain
\begin{multline}\label{EQ_est3}
\mathcal K_r(x)-1=2\sum_{k=1}^{+\infty}\frac{x^k-(r^2/x)^k}{1-r^{2k}}\le
\frac{2}{1-r^2}\sum_{k=1}^{+\infty}\big(x^k-(r^2/x)^k\big)\\=
\frac{2}{1-r^2}\left(\frac{x}{1-x}-\frac{r^2/x}{1-r^2/x}\right)=\frac{2}{1-r^2}\,\frac{(x^2-r^2)/x}{(1-x)(1-r^2/x)}\\\le
\frac{4(x-r)}{(1-r^2)(1-x)(1-r^2/x)}\le
\frac{4(x-r)}{(1-r^2)(1-x)(1-r)} ,\quad 0<r<x<1.
\end{multline}

Applying \eqref{EQ_estReP} and~\eqref{EQ_est2}, from~\eqref{EQ_Drho} we get
\begin{multline*}\label{EQ_}
-\left(\frac{1}{\sqrt{\rho(t)}}+\frac{r(t)}{\rho(t)^{3/2}}\right)\frac{d\rho(t)}{dt}%
\le\\%
-r'(t)\left[\frac{2}{\sqrt{\rho(t)}}+\frac{\nu(t)}{r(t)}F(t)\mathcal
K_{r(t)}\big(\rho(t)\big)+\frac{2\nu(t)}{\sqrt{\rho(t)}}\Big(\mathcal
K_{r(t)}\big(\rho(t)\big)-1\Big)\right],
\end{multline*}
where $F(t):=\sqrt{\rho(t)}-r(t)/\sqrt{\rho(t)}$ for all $t\ge0$. Adding $2r'(t)/\sqrt{\rho(t)}$ to
both sides and applying estimate~\eqref{EQ_est3}, we finally obtain
\begin{equation}\label{EQ_F}
-2\frac{dF(t)}{dt}\le -r'(t)\left[\frac{\nu(t)}{r(t)}F(t)\mathcal
K_{r(t)}\big(\rho(t)\big)+\frac{8\nu(t)F(t)}{(1-r(t)^2)(1-r(t))(1-\rho(t))}\right].
\end{equation}

Recall that $\rho(t)>r(t)>0$ for all $t\ge0$ and that both $r(t)$ and $\rho(t)$ tend to $0$ as
$t\to+\infty$. Hence $F(t)>0$ for all $t\ge0$ and $F(t)\to0$ as $t\to+\infty$. In particular, since
$F$ is continuous, there exists a sequence $(t_n)\subset[0,+\infty)$ tending to $+\infty$ such that
for every $n\in\Natural$, $F(t_n)\ge F(t)$ whenever $t\ge t_n$. Then, from~\eqref{EQ_F} for any
$n\in\Natural$ we obtain
\begin{equation*}\label{EQ_F1}
2F(t_n)\le -F(t_n)\int\limits_{t_n}^{+\infty}r'(t)\left[\frac{\nu(t)}{r(t)}\mathcal
K_{r(t)}\big(\rho(t)\big)+\frac{8\nu(t)}{(1-r(t)^2)(1-r(t))(1-\rho(t))}\right]dt.
\end{equation*}
Thus, bearing in mind that $F(t_n)>0$,
\begin{equation}\label{EQ_F2}
2\le
-\int\limits_{t_n}^{+\infty}r'(t)\left[\frac{\nu(t)}{r(t)}\mathcal
K_{r(t)}\big(\rho(t)\big)+\frac{8\nu(t)}{(1-r(t)^2)(1-r(t))(1-\rho(t))}\right]dt.
\end{equation}

By~\eqref{EQ_est1}, $\mathcal K_{r(t)}\big(\rho(t)\big)$ is bounded.
Recall also that by our assumption, the integral
$\int_0^{+\infty}r'(t)\nu(t)/r(t)\,dt$ converges. Hence the integrals 
in the right-hand side of~\eqref{EQ_F2} (note that they depend on $n$) converge as well and
their values tend to $0$ as $n\to +\infty$. However, this fact contradicts inequality~\eqref{EQ_F2}
for $n$ large enough, which completes the proof of the theorem.
\end{proof}

\section{Semicomplete weak holomorphic vector fields in the mixed-type
case}\label{S_mixed-type}

Consider a canonical domain system $(D_t)=(\mathbb A_{r(t)})$ of some order $d\in[1,+\infty]$.
Recall that $(D_t)$ is called {\it non-degenerate} ({\it degenerate})  if $r(t)>0$ for all $t\ge0$
($r(t)\equiv0$, respectively).  If there exists $T\in(0,+\infty)$ such that $r(t)>0$ for
$t\in[0,T)$ and $r(t)=0$ for $t\ge T$, then we say that $(D_t)$ is of {\it mixed type}.

In~\cite[\S5.1]{SMP3} we established an explicit characterization of semicomplete weak holomorphic
vector fields of order~$d$ over a non-degenerate canonical domain system of the same order~$d$,
similar to the non-autonomous Berkson\,--\,Porta representation in Loewner Theory in the unit
disk~\cite[Theorem~4.8]{BCM1}. The degenerate case was shown to be equivalent to the case of the
unit disk with the common fixed point at the origin~\cite[\S5.2]{SMP3}.

In this section we will combine results mentioned above with Theorem~\ref{TH_EF0} to obtain a
characterization of semicomplete weak holomorphic vector fields in the mixed-type case. To simplify
the formulation of our result we will use notation $\ParClass_0$ for the Carath\'eodory class
consisting, by definition, of all holomorphic functions $p:\UD\to\Complex$ such that $p(0)=1$ and
$\Re p(z)>0$ for all $z\in\UD$.

\begin{theorem}\label{TH_semi-char-mixed-type}
Let $d\in[1,+\infty]$ and let $\big(D_t\big)=\big(\mathbb
A_{r(t)}\big)$ be an $L^d$-canonical domain system of mixed type
with $\mathcal T:=\inf\{t\ge0:r(t)=0\}$. Then a function $G:\mathcal
D\to\Complex$, where $\mathcal D:=\{(z,t):\,t\ge0,\,z\in D_t\}$, is
a semicomplete weak holomorphic vector field of order~$d$ if and
only if there exist functions $\alpha:[0,+\infty)\to[0,+\infty)$,
$C:[0,+\infty)\to\Real$, and $p:\mathcal D\to\Complex$ such that:
\begin{mylist}
\item[(i)] $G(w,t)=w\big[iC(t)-
\alpha(t) p(w,t)\big]$ for a.e. $t\ge0$ and all $w\in D_t$;
\item[(ii)] for each $w\in D:=\cup_{t\ge0} D_t$ the function $p(w,\cdot)$ is measurable in
$E_w:=\{t\ge0:\,(w,t)\in\mathcal D\}$;
\item[(iii)]  for each $t\ge0$ the function $p(\cdot\,,t)$ belongs to the class $\ParClass_{r(t)}$;
\item[(iv)] $C\in L^d_{\mathrm{loc}}\big([0,+\infty),\Real\big)$;
\item[(v)] $\alpha(t)=-r'(t)/r(t)$ for a.e. $t\in[0,\mathcal T)$ and $\alpha|_{[\mathcal T,+\infty)}\in
L^d_{\mathrm{loc}}\big([\mathcal T,+\infty), [0,+\infty)\big)$;
\item[(vi)] the function $t\mapsto P(t):=\alpha(t)\mathcal N\big(p(\cdot,t)\big)$ belongs to $L^d_{\mathrm{loc}}
\big([0,+\infty),\Real\big)$.
\end{mylist}
\end{theorem}

\begin{remark}\label{RM_P(t)}
Condition~(vi) in the above theorem is equivalent, provided
(ii)\,--\,(v) hold, to the requirement that  $P|_{[0,\mathcal T]}\in
L^d \big([0,\mathcal T],\Real\big)$ . Indeed, from $\mathcal
N(\mathcal K_r)=1$ for any $r\in[0,1)$, it follows that
$0\le\mathcal N\big(p(\cdot,t)\big)\le1$ for all $t\ge0$. Condition
(ii) implies that $t\mapsto \mathcal N\big(p(\cdot,t)\big)$ is
measurable in $[0,+\infty)$. Now our claim is clear in view of
condition~(v).
\end{remark}

In the proof of Theorem~\ref{TH_semi-char-mixed-type} we will use a change of variable, which clearly preserves evolution families of any given order. At the same time, the possibility of change of variable might be of some independent interest in principle. The following statement establishes much more general conditions for admissibility of a variable change.
\begin{proposition}\label{PR_tau}
Let $[0,+\infty)\ni t\mapsto \tau(t)\in [0,+\infty)$ be increasing and
locally absolutely continuous. Suppose that the inverse mapping $\tau^{-1}:\tau(I)\to [0,+\infty)$ is also locally absolutely continuous
and that $\tau'\in L_{\mathrm{loc}}^\infty\big([0,+\infty),\Real\big)$.
Then:
\begin{mylist}
\item[(i)] for any $d\in[1,+\infty]$ and any $L^d$-evolution family~$\big((D_t),(\varphi_{s,t})\big)$, the formulas $\varphi^*_{s,t}:=\varphi_{\tau(s),\tau(t)}$ and $D^*_t:=D_{\tau(t)}$, where $0\le s\le t$, define an $L^d$-evolution family $\big((D^*_t), (\varphi^*_{s,t})\big)$;

\item[(ii)] if $G:\mathcal D:=\{(z,t):t\ge0,\,z\in D_t\}\to\Complex$ is a semicomplete weak holomorphic vector field of order~$d$, then $G^*:\mathcal D^*:=\{(z,t):t\ge0,\,z\in D^*_t\}\to\Complex$ defined by $G^*(z,t):=G(z,\tau(t))\tau'(t)$, $t\ge0$, is also  a semicomplete weak holomorphic vector field of order~$d$;

\item[(iii)] if the vector field $G$ generates, in the sense of Theorem~\ref{TH_EFsWHVF}, the evolution family $(\varphi_{s,t})$, then the vector field $G^*$ generates, in the same sense, the evolution family $(\varphi^*_{s,t})$.
\end{mylist}
\end{proposition}

\begin{proof}
First of all note that since  both $\tau$ and $\tau^{-1}$ are locally absolutely continuous, $\tau(I)$ and $f\circ\tau$ are measurable for any measurable set $I\subset[0,+\infty)$ and any measurable function~$f:\tau(I)\to\Real$. Moreover, we claim that

\noindent\underline{Claim}{. \it For any $d\in[1,+\infty]$, any measurable set $I\subset[0,+\infty)$, and any $f\in L^d\big(\tau(I),\Real\big)$, the function $(f\circ \tau)\,\tau'$ belongs to $L^d(I,\Real)$.}

Indeed, since $\tau^{-1}$ is locally absolutely continuous,  $\|f\circ\tau\|_{L^\infty(I)}\le \|f\|_{L^\infty(\tau(I))}$. This proves our claim for~$d=+\infty$. Now assume $d\in[1,+\infty)$. Making change of variable in the integral and using the H\"older inequality, we get
\begin{eqnarray*}
\|(f\circ \tau)\,\tau'\|^d_{L^d(I)}&=&\int_I\big|f(\tau(t))\big|^d\big(\tau'(t)\big)^d\,dt\\
& = &
\int_{\tau(I)}\big|f(\xi)\big|^d\big(\tau'(\tau^{-1}(\xi))\big)^{d-1}\,d\xi\\
&\le &
\|f\|^d_{L^d(\tau(I))}\cdot\|\tau'\circ\tau^{-1}\|_{L^\infty(\tau(I))}^{d-1}.
\end{eqnarray*}
Since $\tau$ is locally absolutely continuous, $\|\tau'\circ\tau^{-1}\|_{L^\infty(\tau(I))}\le \|\tau\|_{L^\infty(I)}<+\infty$. This completes the proof of the claim for $d\in[0,+\infty)$.

Now let us prove (i). Assume $(D_t)=(\mathbb A_{r(t)})$ is a canonical domain system of order~$d$, see Definition~\ref{def-cansys}. Then the function $t\mapsto \omega(r(t))$ is of class $AC^d\big([0,+\infty),\Real\big)$. From the above claim it follows that $t\mapsto \omega\big(r(\tau(t))\big)$ is also of class $AC^d\big([0,+\infty),\Real\big)$, and hence $(D^*_t)$ is also a canonical domain system of order~$d$.
Furthermore, it is evident that the family $(\varphi^*_{s,t})$ satisfies conditions EF1 and EF2 from Definition~\ref{def-ev}. So we only need to prove~EF3 for~$(\varphi^*_{s,t})$. Fix $I:=[S,T]\subset[0,+\infty)$ and $z\in D^*_S$. By EF3 for $(\varphi_{s,t})$ with $\tau(S)$ and $\tau(T)$ substituted for $S$ and $T$, respectively, there exists a
non-negative function ${k_{z,\tau(I)}\in L^{d}\big(\tau(I),\mathbb{R}\big)}$
such that
\begin{equation}\label{EQ_forphi}
|\varphi_{\tau(s),\tau(u)}(z)-\varphi_{\tau(s),\tau(t)}(z)|\leq\int_{\tau(u)}^{\tau(t)}k_{z,\tau(I)}(\xi)d\xi
\end{equation}
whenever $S\leq s\leq u\leq t\leq T.$ Now let $k^*_{z,I}:=(k_{z,\tau(I)}\circ\tau)\circ\tau'$. By our claim, $k^*_{z,I}\in L^{d}\big(I,\mathbb{R}\big)$. Then, using the change of variable $\xi=\tau(\sigma)$, from~\eqref{EQ_forphi} we get
$$
|\varphi^*_{s,u}(z)-\varphi^*_{s,t}(z)|\leq\int_{s}^{t}k^*_{z,I}(\sigma)d\sigma,
$$ which proves~EF3 for~$(\varphi^*_{s,t})$.

To prove~(ii) we observe first that $G^*$ is a weak holomorphic vector field of order~$d$. Indeed, conditions WHVF1 and WHVF2 in Definition~\ref{D_WHVF} hold trivially, while WHVF3 holds with $(k_{\tau_*(K)}\circ\tau)\,\tau'$ substituted for~$k_K$, where $\tau_*:\big(z,t\big)\mapsto\big(z,\tau(t)\big)$ and $k_{\tau_*(K)}$ is the function of class $L^d$ from condition WHVF3 for the original vector field~$G$. The fact that $G^*$ is semicomplete follows from the fact that $G$ is semicomplete and that if $t\mapsto w(t)$  solves the equation $dw/dt=G(w,t)$ then the function $w^*:=w\circ\tau$ is a solution to $dw^*/dt=G^*(w^*,t)$. By the same reason, (iii) takes place. Thus the proof is complete.
\end{proof}

\begin{proof}[{\bf Proof of Theorem~\ref{TH_semi-char-mixed-type}}]
Let us prove first that conditions of the theorem are necessary for
$G$ to be a semicomplete weak holomorphic field of order~$d$. By
Theorem~\ref{TH_EFsWHVF}, any semicomplete weak holomorphic field of
order~$d$ over $(D_t)$ generates an evolution family
$(\varphi_{s,t})$ over $(D_t)$ of the same order~$d$ such that
$(d/dt)\varphi_{s,t}(z)=G\big(\varphi_{s,t}(z),t\big)$ for any
$s\ge0$ and a.e. $t\ge s$.

According to Proposition~\ref{PR_tau}(i) the formulas
\begin{equation}
\varphi_{s,t}^1:=\varphi_{\tau(s),\tau(t)},\quad \tau(t):=\mathcal T(1-e^{-t}),\qquad
\varphi_{s,t}^2:=\varphi_{s+\mathcal T,t+\mathcal T}
\end{equation}
define two evolution families of order $d$: $(\varphi^1_{s,t})$ is
an evolution family over the non-degenerate $L^d$-canonical domain
system $(D_t^1):=(D_{\tau(t)})$ and $(\varphi^2_{s,t})$ is an
evolution family over the degenerate canonical domain system
$(D_t^2):=(D_{t+\mathcal T})$. Moreover, the semicomplete weak
holomorphic vector fields $G^1$ and $G^2$ corresponding in the sense
of Theorem~\ref{TH_EFsWHVF} to the evolution families
$(\varphi_{s,t}^1)$ and $(\varphi_{s,t}^2)$, respectively, are given
for a.e. $t\ge0$ by
\begin{equation}
G^1(z,t)=\frac{d\varphi_{s,t}^{1}(z)}{dt}|_{s:=t}=G(z,\tau(t))\mathcal T e^{-t},\quad
G^2(z,t)=\frac{d\varphi_{s,t}^{2}(z)}{dt}|_{s:=t}=G(z,t+\mathcal T).
\end{equation}
Hence, using Theorem~5.6 and Proposition~5.16 from~\cite{SMP3}, one can conclude that there exist
functions $\alpha:[0,+\infty)\to[0,+\infty)$, $C:[0,+\infty)\to\Real$, and $p:\mathcal
D\to\Complex$ satisfying conditions~(i)\,--\,(iii) and (v).

To prove (iv) we use essentially the same argument as in~\cite{SMP3}. Note that ${\mathcal
N(\mathcal K_r)=1}$ for any $r\in[0,1)$. Hence $\Im \mathcal N(p(\cdot,t))=0$ for all $t\ge0$ and
consequently $C(t)=(1/2\pi)\Im\int_\UC G(\rho\xi,t)/(\rho\xi)\,{|d\xi|}$, where we have fixed some
$\rho\in\big(r(0),1\big)$. Since by definition $G(\cdot,z)$ is measurable for all $z\in \mathbb
A_{r(0)}$ and for every $T>0$ there exists a non-negative $k_T\in L^d\big([0,T],\Real\big)$ such
that $|G(z,t)|\le k_T(t)$ whenever $|z|=\rho$ and $t\in[0,T]$, it follows with the help of the
Lebesgue dominated convergence theorem that $t\mapsto C(t)$ belongs to $L^d_{\rm
loc}\big([0,+\infty),\Real\big)$.

To check condition~(vi) fix any $z_0\in D_0$ and denote $\rho(t):=|\varphi_{0,t}(z_0)|$ for all
$t\ge0$. By continuity of $t\mapsto\rho(t)$ and $t\mapsto r(t)$ there exists $\delta>0$ such that
$r(t)+\delta<\rho(t)<1-\delta$ for all $t\in[0,\mathcal T]$. Repeating the argument to deduce
inequality~\eqref{EQ_nuzhno} in the proof of Theorem~\ref{TH_EF0}, we can conclude that for any $S$
and $T$ satisfying $0\le S\le T\le \mathcal T$,
\begin{equation}\label{EQ_nuno}
\rho(S)-\rho(T)\ge -\frac{\delta^2}{2}\int_S^T r'(t)\frac{\nu(t)}{r(t)}dt+\frac2{\delta}\int_S^T
\frac{r'(t)}{1-r(t)^2}dt,
\end{equation}
where $\nu(t):=\mathcal N\big(p(\cdot,t)\big)$ for all $t\in[0,\mathcal T]$. From the definition of
a canonical domain system of order $d$ and the definition of an $L^d$-evolution family it follows
that the functions $t\mapsto r(t)$ and $t\mapsto \rho(t)$ belong to $AC^d\big([0,\mathcal
T],\Real\big)$. Hence from~\eqref{EQ_nuno} it follows that the function $t\mapsto
-r'(t)\nu(t)/r(t)$ belongs to $L^d\big([0,\mathcal T],\Real\big)$. In view of Remark~\ref{RM_P(t)},
this means that (vi) holds.

Thus we have proved that conditions of the theorem are necessary for $G$ to be a semicomplete weak
holomorphic field of order~$d$. Let us now show that they are also sufficient.

Assume that there exist functions $\alpha:[0,+\infty)\to[0,+\infty)$, $C:[0,+\infty)\to\Real$, and
$p:\mathcal D\to\Complex$ satisfying conditions~(i)\,--\,(vi). The first step to prove that $G$ is
a semicomplete weak holomorphic vector field of order~$d$,  is to check conditions WHVF1, WHVF2,
and WHVF3 from Definition~\ref{D_WHVF}. The first two of these conditions follow directly from
(i)\,--\,(v) if it is taken into account that $t\mapsto r'(t)/r(t)$ is measurable in $[0,\mathcal T)$
according to the definition of a canonical domain system.

Let us show that $G$ satisfies also WHVF3. To this end we have to obtain some estimates for the
Villat kernel. Using its Laurent development, we in particular get
\begin{multline*}
|\mathcal K_r(z)-1|=2\left|\sum_{k=1}^{+\infty}\frac{z^k-(r^2/z)^k}{1-r^{2k}}\right|\le
\frac{2}{1-r^2}\sum_{k=1}^{+\infty}\big(|z|^k+|r^2/z|^k\big)\\ \le
\frac{2}{1-r^2}\left(\frac{|z|}{1-|z|}+\frac{r^2/|z|}{1-r^2/|z|}\right),\quad 0<r<|z|<1,
\end{multline*}
It follows that
\begin{eqnarray*}
&\displaystyle|\mathcal K_r(z)|\le A(r,|z|):=1+\frac{2}{1-r^2}\left(\frac{|z|}{1-|z|}+\frac{r}{|z|-r}\right),&\\
&\displaystyle|\mathcal K_r(r/z)-1|\le
B(r,|z|):=\frac{2r}{1-r^2}\left(\frac{1}{1-|z|}+\frac{1}{|z|-r}\right),&\quad 0<r<|z|<1,
\end{eqnarray*}
from which we deduce, using Remark~\ref{RM_int}, that for all $t\in[0,\mathcal T)$,
\begin{equation}\label{EQ_est_p}
|p(z,t)|\le A(r(t),|z|)\mathcal N\big(p(\cdot,t)\big)+B(r(t),|z|)\big[1-\mathcal
N\big(p(\cdot,t)\big)\big].
\end{equation}
The above inequality holds also for $t\ge \mathcal T$, because then $r(t)=0$ and consequently
$$\mathcal N\big(p(\cdot,t)\big)=1,\quad A(r(t),|z|)=A(0,|z|)=(1+|z|)/(1-|z|),$$ and \eqref{EQ_est_p} reduces to
the well-known estimate for the Carath\'eodory class, see e.g.~\cite[ineq.\,(11) on
p.\,40]{Pommerenke}.

We proceed as in the proof of~\cite[Lemma~5.14]{SMP3}. Let us fix any compact set $K\subset\mathcal
D$. Then there exists $\delta>0$ and $T>0$ such that $r(t)+\delta\le|z|\le1-\delta$ and $t\le T$
for all $(z,t)\in K$. From (i) and~\eqref{EQ_est_p} we deduce that

\begin{equation}\label{EQ_G_est}
|G(z,t)|\le |C(t)|+(1+M)\alpha(t)\mathcal N\big(p(\cdot,t)\big)+M\alpha(t)r(t)\quad\text{for
all}~(z,t)\in K,
\end{equation}
where $$M:=\frac{4/\delta}{1-r(0)^2}.$$ Recall that $t\mapsto r'(t)$ belong to $L^d\big([0,\mathcal
T],\Real\big)$. Hence from~(iv)\,--\,(vi) it follows that the right-hand side of~\eqref{EQ_G_est}
is a function of $t$ from $L_{\rm loc}^d\big([0,+\infty),\Real\big)$. This completes that proof
of~WHVF3.

It remains to show that the weak holomorphic vector field $G$ is semicomplete.
To this end we write:\def\Tau{\mathcal T}
$$G^1(w,t):=G(w,\tau(t))\Tau e^{-t},\quad t\ge0,~~w\in D_{\tau(t)},$$
where $\tau(t):=\Tau (1-e^{-t})$, and
$$
G^2(w,t):=G(w,t+\Tau),\quad t\ge0,~~w\in D_{t+\Tau}.
$$

By Theorem~5.6 and Proposition~5.18 from~\cite{SMP3}, the functions $G^1$ and
$G^2$ are semicomplete weak holomorphic vector fields of order~$d$ in $\mathcal
D_1:=\{(w,t):t\ge0,\,w\in D_{\tau(t)}\}$ and $\mathcal
D_2:=\{(w,t):t\ge0,\,w\in D_{t+\Tau}\}$ respectively. Therefore,
by~\cite[Theorem~5.1(B)]{SMP3}, there exist unique evolution families
$\big(\varphi^1_{s,t}\big)$ and $\big(\varphi^2_{s,t}\big)$ over the canonical
domain systems $(D_{\tau(t)})$ and $(D_{t+\Tau})$ respectively such that
$(\partial/\partial t)\varphi^1_{s,t}(z_1)=G^1\big(\varphi^1_{s,t}(z_1),t\big)$
 and $(\partial/\partial t)\varphi^2_{s,t}(z_2)=G^2\big(\varphi^2_{s,t}(z_2),t\big)$ for all
 $s\ge0$, a.e. $t\ge s$ and any $z_1\in D_{\tau(t)}$, $z_2\in D_{t+\Tau}$.

 By construction, the equation $dw/d\tau=G(w,\tau)$ is equivalent to $d
 w(\tau(t))/dt=G^1\big(w(\tau(t)),t\big)$ when $\tau\in[0,\Tau)$ and to the
 equation $dw(t+\Tau)/dt=G^2\big(w(t+\Tau),t\big)$ when $\tau\ge\Tau$. It
 follows now from the general theory of Carath\'eodory ODEs (see, e.g., \cite[Theorem
 2.3]{SMP3}) that it is sufficient to show that given $s\ge0$ and $z\in
 D_{\tau(s)}$ there exists $\delta=\delta(z,s)>0$ such that
 $r(\tau(t))+\delta<|\varphi^1_{s,t}(z)|<1-\delta$ for all $t\ge s$.
 Recall  that $t\mapsto r(\tau(t))$ and $t\mapsto |\varphi^1_{s,t}(z)|$ are
 continuous  functions in~$[s,+\infty)$, with  $r(\tau(t))<|\varphi^1_{s,t}(z)|<1$ for all~$t\ge0$.
 Hence according to Lemma~\ref{LM_ML}, it
 remains to see that $\varphi^1_{0,t}\not\to0$ as $t\to+\infty$. The latter follows
 readily from~(vi) and Theorem~\ref{TH_EF0}. The proof is complete.
\end{proof}

\section{Conformal types of Loewner chains via vector fields}\label{S_conf-type-VF}

In this section we combine results of
Sections~\ref{S_conf_type_EF}\,--\,\ref{S_mixed-type} to get the conformal
classification of Loewner chains in terms of the corresponding vector fields.

Let $d\in[1,+\infty]$. Let $(D_t)=\big(\mathbb A_{r(t)}\big)$ be a canonical
domain system of order~$d$ and $(f_t)$ a Loewner chain of order~$d$
over~$(D_t)$. By Corollary~\ref{C_from_LC_to_VF} there exists a unique vector
field $G$ associated with~$(f_t)$, i.e. a weak holomorphic vector field $G$ of
order~$d$ over~$(D_t)$ satisfying for a.e. $s\ge0$ equality~\eqref{EQ_GK-PDE}.

Let us assume first that $(D_t)$ is non-degenerate, i.e.
$D_t:=\mathbb A_{r(t)}$ with $r(t)>0$ for each $t\ge0$. The
following theorem characterizes the conformal type of $(f_t)$ via
the associated vector field $G$.

Recall (see Remark~\ref{RM_int}) that for a.e. $t\ge0$, the function
$G_t=G(\cdot,t)$ admits the following representation
\begin{equation*}
G_t(z) = z\left(\frac{r'(t)}{r(t)}p_t(z)+iC_t\right),
\end{equation*}
where $C_t\in\Real$,
\begin{equation*}
p_t(z):=\int_\UC\mathcal
K_r(z/\xi)d\mu^t_1(\xi)+\int_\UC\big[1-\mathcal
K_r(r\xi/z)\big]d\mu^t_2(\xi),\quad z\in\mathbb A_r,~r:=r(t),
\end{equation*}
and $\mu_1^t$ and $\mu_2^t$ are positive Borel measures on the unit
circle~$\UC$ subject to the condition~$\mu_1^t(\UC)+\mu_2^t(\UC)=1$.
Denote
\begin{align*}
I_1&:=-\int\limits_{0}^{+\infty}\Re\mathcal
 N\big(w\mapsto G_t(w)/w\big)\,dt=-\int\limits_{0}^{+\infty}r'(t)\frac{\mu_1^t(\UC)}{r(t)}dt,\quad\\
I_2&:=-\int\limits_{0}^{+\infty}\left(\frac{r'(t)}{r(t)}-\Re
\mathcal
N\big(w\mapsto G_t(w)/w\big)\right)dt=-\int\limits_{0}^{+\infty}r'(t)\frac{\mu_2^t(\UC)}{r(t)}dt.\\
\end{align*}

\begin{theorem}\label{TH_conf-class-VF-nondeg}
In the above notation, the following statements hold:
\begin{itemize}
  \item[(i)] the evolution family $(\varphi_{s,t})$ is of type I if and only
  if $I_1+I_2<+\infty$;
  \item[(ii)] the evolution family $(\varphi_{s,t})$ is of type II if and only if $I_1<+\infty$ and $I_2=+\infty$;
  \item[(iii)] the evolution family $(\varphi_{s,t})$ is of type III if and only if $I_1=+\infty$ and $I_2<+\infty$;
  \item[(iv)] the evolution family $(\varphi_{s,t})$ is of type IV if and only if
  $I_1=I_2=+\infty$.
\end{itemize}
\end{theorem}
\begin{proof}[{\bf Proof of Theorem~\ref{TH_conf-class-VF-nondeg}}]
Note that $I_1,I_2\ge0$ and that $I_1+I_2<+\infty$ if and only if
$r_\infty:=\lim_{t\to+\infty}r(t)>0$.

Let $(\varphi_{s,t})$ be the evolution family of the Loewner
chain~$(f_t)$ and let $\tilde
\varphi_{s,t}(z):=r(t)/\varphi_{s,t}(r(s)/z)$ for all $s\ge0$, all
$t\ge s$ and all $z\in D_s$. Then by Theorem~\ref{TH_EF0}, combined
with Remark~\ref{RM_int}, $\varphi_{0,t}\to 0$ as $t\to+\infty$ if
and only if $I_1=+\infty$. Similarly, with the help
of~\cite[Example~6.3]{SMP3} we conclude that $\tilde\varphi_{0,t}\to
0$ as $t\to+\infty$ if and only if $I_2=+\infty$.

Now Theorem~\ref{TH_conf-class-VF-nondeg} follows immediately from
Theorem~\ref{TH_types_nondeg}.
\end{proof}

Assume now that the canonical domain system $(D_t)$ is either
degenerate or of mixed type. In other words, we suppose that
$\{t\ge0:r(t)=0\}\neq\emptyset$. Let us denote
$$
I:=-\int\limits_{0}^{+\infty}\Re\mathcal
 N\big(w\mapsto G(w,t)/w\big)\,dt.
$$

\begin{proposition}\label{PR_conf-class-VF-mixed}
In the above notation, the Loewner chain~$(f_t)$ is of type II if
$I<+\infty$ and of type~$IV$ otherwise.
\end{proposition}
\begin{proof}
By Proposition~\ref{TH_types_mixed} it is sufficient to show that
$\varphi_{0,t}\to0$ as $t\to+\infty$ if and only if $I=+\infty$.

Choose $T>0$ such that $r(T)=0$. From \cite[Proposition 5.15]{SMP3}
it follows that the family $(\phi_{s,t})_{0\le s\le t}$ defined by
$\phi_{s,t}(z):=\varphi_{s+T,t+T}(z)$ for all $z\in \UD^*$ and by
$\phi_{s,t}(0)=0$, is an $L^d$-evolution family in the unit disk.
Hence the vector field $\mathcal G(w,t):=G(w,t+T)$, $t\ge0$,
$w\in\UD^*$, extended to the origin by $\mathcal G(0,t)=0$, is a
Herglotz vector field in~$\UD$ and $d\phi_{s,t}(z)/dt=\mathcal
G(\phi_{s,t}(z),t)$ for all $z\in\UD$ and a.e. $t\ge 0$. It follows
that $$|\phi_{0,t}'(0)|=\exp\int_0^t \Re \mathcal
G'(0,\xi)d\xi=\exp\left(-\int_0^t\Re\mathcal
 N\big(w\mapsto G(w,\xi)/w\big)d\xi\right).$$
Thus $\phi_{0,t}'(0)\to0$ as $t\to+\infty$ if and only if
$I=+\infty$. This completes the proof.
\end{proof}

\end{document}